\documentclass[12pt, reqno]{amsart}

\usepackage{url}
\usepackage{hyperref}
\usepackage{pdfsync}
\usepackage{amssymb}
\usepackage{amsmath}
\usepackage{amsthm}
\usepackage{stmaryrd}
\usepackage{fullpage}	
\usepackage{amscd}   	
\usepackage[all]{xy} 		
\usepackage{comment} 	
\usepackage{mathrsfs}      

\usepackage{enumerate}

\usepackage{tikz}
\usetikzlibrary{decorations.pathmorphing}

\usepackage[alphabetic,backrefs,lite]{amsrefs} 

\usepackage{color}

\newcommand{\marg}[1]{}

\newcommand{\note}[1]{} 				

\newcommand{\defi}[1]{\textsf{#1}} 				

\newcommand{\C}{{\mathbb C}}
\newcommand{\F}{{\mathbb F}}

\newcommand{\mm}{{\mathfrak m}}

\newcommand{\calM}{{\mathcal M}}

\newcommand{\calO}{{\mathcal O}}

\newcommand{\calS}{{\mathcal S}}

\newcommand{\calV}{{\mathcal V}}

\newcommand{\calX}{{\mathcal X}}
\newcommand{\calY}{{\mathcal Y}}

\newcommand{\scrH}{{\mathscr H}}

\def\C{\mathbb{C}}

 \DeclareMathOperator{\Hom}{Hom}
\DeclareMathOperator{\HOM}{\scrH om}

\DeclareMathOperator{\Spec}{Spec}

\DeclareMathOperator{\Mod}{Mod}
\DeclareMathOperator{\Open}{Open}
\DeclareMathOperator{\Coh}{Coh}
\DeclareMathOperator{\id}{id}

\DeclareMathOperator{\Th}{th}

\newcommand{\Sch}{\operatorname{\bf Sch}}
\newcommand{\Sets}{\operatorname{\bf Sets}}
\newcommand{\Ab}{\operatorname{Ab}}

\newcommand{\op}{{\operatorname{op}}}
\newcommand{\an}{{\operatorname{an}}}
\newcommand{\ab}{{\operatorname{ab}}}

\DeclareMathOperator{\MM}{M}

\DeclareMathOperator{\AN}{AN} 
\DeclareMathOperator{\Spf}{Spf}
\DeclareMathOperator{\Isoc}{Isoc}
\DeclareMathOperator{\QCoh}{QCoh}
\DeclareMathOperator{\fp}{fp}
\DeclareMathOperator{\g}{g}
\DeclareMathOperator{\e}{e}
\DeclareMathOperator{\rig}{rig}

\DeclareMathOperator{\Fun}{Fun}
\DeclareMathOperator{\colim}{colim}
\DeclareMathOperator{\Cov}{Cov}
\DeclareMathOperator{\set}{set}
\DeclareMathOperator{\Ber}{Ber}
\DeclareMathOperator{\Strat}{Strat}
\DeclareMathOperator{\MIC}{MIC}

\DeclareMathOperator{\Cris}{Cris}

\DeclareMathOperator{\Mor}{Mor}

 \numberwithin{equation}{subsection}
\newtheorem{theorem}[subsection]{Theorem}
\newtheorem{lemma}[subsection]{Lemma}
\newtheorem{corollary}[subsection]{Corollary}
\newtheorem{proposition}[subsection]{Proposition}

\theoremstyle{definition}
\newtheorem{definition}[subsection]{Definition}

\newtheorem{example}[subsection]{Example}

\theoremstyle{remark}
\newtheorem{remark}[subsection]{Remark}

\newcommand{\labelpar}[1]{\refstepcounter{subsection}\label{#1}\thesubsection.}

\begin{document}

\title[Cohomology with closed support on the overconvergent site]
         {Cohomology with closed support on the overconvergent site}

\author{David Zureick-Brown}
\address{Dept. of Mathematics and Computer science, Emory University,
Atlanta, GA 30322 USA}

\date{\today}

\begin{abstract}
Using the notions of open/closed subtopoi of SGA, we define a notion of
cohomology with support in a   closed subscheme on the overconvergent site, and show
that this agrees with the classic notion of rigid cohomology support in a   closed subscheme.
\end{abstract}

\maketitle



\section{Introduction}
\label{C:introduction}

The pursuit of a Weil, or `topological', cohomology theory in algebraic geometry was a driving factor in the development of Grothendieck's notion of a scheme and the subsequent ideas which permeate modern algebraic geometry and number theory. The initial success was \'etale cohomology and the subsequent proof of the Weil conjectures -- that for a prime power $q = p^r$ and a variety $X$ over the finite field $\F_q$, the numbers $X(\F_{q^n})$ of $\F_{q^n}$-points of $X$ (i.e., the number of $\F_{q^n}$-valued solutions to the polynomial equations defining $X$) are governed by strict and surprising formulas: they depend on the dimensions of the singular cohomology $H^i(X'(\C); \C)$ spaces of a lift $X'$ of $X$ to characteristic zero (when, of course, such a lift exists).

Applications abound (see e.g. \cite{Illusie:crystalSurvey}*{1.3.8},  \cite{Illusie:deRhamWitt}*{II.7.3},  \cite{Kedlaya:zeta}, \cite{Kedlaya:hyperellipticCounting}*{Section 5}, \cite{abbotKR:picard}). Le Stum's recent foundational advance \cite{leStum:site} characterizes rigid cohomology as the cohomology of a site $\AN^{\dagger}(X)$, removing many of the technical foundations of the theory and  leading to simple proofs of once difficult theorems and allowing new extensions and applications of rigid cohomology.

Cohomology \emph{supported in a closed subspace} plays a key role in the theory of rigid cohomology (e.g., in Kedlaya's proof of finite dimensionality of rigid cohomlogy with coefficients \cite{kedlaya:finitenessCoefficients}) and is a natural first piece of technology to adapt to the overconvergent site.  Our main result is an extension of le Stum's work to cohomology with support in a closed subscheme. We define, for a closed immersion $Z \hookrightarrow X$, the overconvergent cohomology of $X$ with supports in $Z$ and prove that it agrees with the classical, concrete definitions due to Berthelot.

\begin{theorem}
\label{P:agreementClosedSupportsCohomology}

Let the assumptions be as in Corollary \ref{P:agreementClosedSupportsIsoc}. Let $p_{\AN^{\dagger}}\colon \AN^{\dagger}X_{P_K}/S_K \to \AN^{\dagger}(S_k,S_K)$ be the induced morphism of sites. Then there is a functorial isomorphism
\[
\left(\mathbb{R}p_{\rig}\underline{\Gamma}^{\dagger, \Ber}_Z E_0\right)^{\an} \cong \left(\mathbb{R}p_{\AN^{\dagger}_*}\underline{\Gamma}^{\dagger}_ZE\right)_{S_k,S_K}
\]
which is compatible with the excision exact sequence of Proposition \ref{P:excisionRelationsDagger} (ii) (and its rigid analogue \cite{leStum:rigidBook}*{Proposition 6.3.9}).

\end{theorem}    

We also prove finiteness, functorality, and excision results and generalize everything to stacks; see Propositions \ref{P:functoralityClosedSupport}, \ref{P:excisionRelationsDagger}, and \ref{P:finiteness}.
\\

The main technical insight of this paper that the very general notions open and closed subtopoi of \cite{SGA4:I}*{expos\'e iv, Section 9} give a natural formalism to define cohomology supported in a closed subspace; functorality and excision then follow from the general theory of \cite{SGA4:I}, and the main work is to show that this agrees with the classical, concrete definitions due to Berthelot.

\subsection{Organization of the paper}
\label{SS:organization}

This paper is organized as follows. In Section \ref{S:conventions} we recall notation. In Sections \ref{S:overconvergentSite} and \ref{S:overconvergentCalculus} we recall the construction of the overconvergent site and of modules with integrable connection of \cite{leStum:site}. (We also note that Section \ref{S:overconvergentSite} is an expansion of \cite{Zureick-Brown2014}*{Section 4}.) In Section \ref{S:topoiReview} we recall the machinery of open and closed subtopoi of \cite{SGA4:I} and prove Theorem \ref{P:agreementClosedSupportsCohomology}. Section \ref{S:main} finishes with an application to the cohomology of stacks.

\subsection{Acknowledgements}
\label{SS:acknowledgements}

We would like to thank Bjorn Poonen, Brian Conrad, Arthur Ogus, Bernard le Stum, Bruno Chiarellotto, and Anton Geraschenko for many useful conversations and encouragement. This work formed part of the author's thesis \cite{brown:thesis}, which used  \cite{leStum:site}  to generalize rigid cohomology to algebraic stacks over fields of positive characteristic.

\section{Notations and conventions}
\label{S:conventions}

Throughout $K$ will denote a field of characteristic 0 that is complete with respect to a non-trivial non-Archimedean valuation with valuation ring $\calV$, whose maximal ideal and residue field we denote by $\mm$ and $k$. We denote the category of schemes over $k$ by $\Sch_k$. We define an \defi{algebraic variety over} $k$ to be a scheme such that there exists a \emph{locally finite} cover by schemes of finite type over $k$ (recall that a collection $\calS$ of subsets of a topological space $X$ is said to be locally finite if every point of $X$ has a neighborhood which only intersects finitely many subsets $X \in \calS$). Note that we do not require an algebraic variety to be reduced, quasi-compact, or separated.
\\

\defi{Formal Schemes}: As in \cite{leStum:site}*{1.1} we define a formal $\calV$-scheme to be a locally topologically finitely presented formal scheme $P$ over $\calV$, i.e.,  a formal scheme $P$  with a locally finite covering by formal affine schemes $\Spf A$, with $A$ topologically of finite type (i.e., a quotient of the ring $\calV\{T_1,\cdots,T_n\}$ of convergent power series by an ideal $I + \mathfrak{a}\calV\{T_1,\cdots,T_n\}$, with $I$ an ideal of $\calV\{T_1,\cdots,T_n\}$ of finite type and $\mathfrak{a}$ an ideal of $\calV$). This finiteness property is necessary to define the `generic fiber' of a formal scheme. 

We refer to \cite{EGAI}*{1.10} for basic properties of formal schemes. The first section of \cite{Berkovich:contractiblity} is another good reference. Actually, \cite{leStum:site}*{Section 1} contains everything we will need.
\\

\defi{$K$-analytic spaces}: We refer to \cite{Berkovich:nonArchEtaleCoh} (as well as the brief discussion in \cite{leStum:site}*{4.2}) for definitions regarding $K$-analytic spaces. As in \cite{leStum:site}*{4.2}, we define an \defi{analytic variety} over $K$ to be a locally Hausdorff topological space $V$ together with a maximal affinoid atlas $\tau$ which is locally defined by \emph{strictly} affinoid algebras. In Section \ref{S:overconvergentCalculus} we collect and review necessary facts from $K$-analytic geometry, and in particular we note that an analytic variety $V$ has a Grothendieck topology which is finer than its usual topology, which we denote by $V_G$ and refer to as `the $G$-topology' on $V$.
\\

\defi{Topoi}:  We follow the conventions of \cite{SGA4:I} (exposited in \cite{leStum:site}*{4.1}) regarding sites, topologies, topoi, and localization; see subsection \ref{S:topoiReview} for a review. When there is no confusion we will identify an object of a category with its associated presheaf. For a topos $T$ we denote by $\mathbb{D}_+(T)$ the derived category of bounded below complexes of objects of $\Ab T$. Often a morphism $(f^{-1},f_*)\colon (T,\calO_T) \to (T',\calO_{T'})$ of ringed topoi will satisfy $f^{-1}\calO_{T'} = \calO_T$, so that there is no distinction between the functors $f^{-1}$ and $f^*$; in this case, we will write $f^*$ for both.

\section{The overconvergent site}
\label{S:overconvergentSite}

Following \cite{leStum:site}, we make the following series of definitions; see \cite{leStum:site} for a more detailed discussion of the definitions with some examples.

\begin{definition}[\cite{leStum:site}, 1.2] 
\label{D:overconvergentVariety}

 Define an \defi{overconvergent variety} over $\calV$ to be a pair $(X \subset P, V \xrightarrow{\lambda} P_{K})$, where $X \subset P$ is a locally closed immersion of an algebraic variety $X$ over $k$ into the special fiber $P_k$ of a formal scheme $P$ (recall our convention that all formal schemes are topologically finitely presented over $\Spf \calV$), and $V \xrightarrow{\lambda} P_K$ is a morphism of analytic varieties, where $P_K$ denotes the generic fiber of $P$, which is an analytic space (in contrast to the Raynaud generic fiber, which is a rigid analytic space; see \cite{leStum:site}*{Section 4.2}). When there is no confusion we will write $(X,V)$ for $(X \subset P, V \xrightarrow{\lambda} P_{K})$ and $(X,P)$ for $(X \subset P, P_K \xrightarrow{\id} P_{K})$.
  Define a \defi{formal morphism} $(X',V') \to (X,V)$ of overconvergent varieties to be a commutative diagram
\[
\xymatrix{
X' \ar@{^(->}[r] \ar[d]^f & P'  \ar[d]^v &  P_K' \ar[l] \ar[d]^{v_K} & V' \ar[l] \ar[d]^u\\
X \ar@{^(->}[r] & P  &  P_K \ar[l] & V \ar[l]
}
\]   
where $f$ is a morphism of algebraic varieties, $v$ is a morphism of formal schemes, and $u$ is a morphism of analytic varieties.

  Finally, define $\AN(\calV)$ to be the category whose objects are overconvergent varieties and morphisms are formal morphisms. We endow $\AN(\calV)$ with the \defi{analytic topology}, defined to be the topology generated by families $\{(X_i,V_i) \to (X,V)\}$ such that for each $i$, the maps $X_i \to X$ and $P_i \to P$ are the identity maps, $V_i$ is an open subset of $V$, and $V = \bigcup V_i$ is an open covering (recall that an open subset of an analytic space is admissible in the $G$-topology and thus also an analytic space -- this can be checked locally in the $G$-topology, and for an affinoid this is clear because there is a basis for the topology of open affinoid subdomains).

\end{definition}

\begin{definition}[\cite{leStum:site}, Section 1.1] 

 The specialization map $P_K \to P_k$ induces by composition a map $V \to P_k$ and we define the \defi{tube} $]X[_V$ of $X$ in $V$ to be the preimage of $X$ under this map. The tube $]X[_{P_K}$ admits the structure of an analytic space and the inclusion $i_X\colon ]X[_{P_K}\, \hookrightarrow P_K$ is a locally closed immersion of analytic spaces (and generally not open, in contrast to the rigid case). The tube $]X[_V$ is then the fiber product $]X[_{P_K}\times_{P_K}V$ (as analytic spaces) and in particular is also an analytic space.

\end{definition}

\begin{remark}
  A formal morphism $(f,u)\colon (X', V') \to (X, V)$ induces a morphism $]f[_u\colon ]X'[_{V'} \to ]X[_V$ of tubes. Since $]f[_u$ is induced by $u$, when there is no confusion we will sometimes denote it by $u$.
\end{remark}

The fundamental topological object in rigid cohomology is the tube $]X[_{V}$, and most notions are defined only up to neighborhoods of $]X[_{V}$. We immediately make this precise by modifying $\AN(\calV)$.

\begin{definition}[\cite{leStum:site}, Definition 1.3.3]
  	Define a formal morphism 
\[
(f,u)\colon (X', V') \to (X, V)
\] 
to be a \defi{strict neighborhood} if $f$ and $]f[_u$ are isomorphisms and $u$ induces an isomorphism from $V'$ to a neighborhood $W$ of $]X[_V$ in $V$.

\end{definition}

\begin{definition}
\label{D:overconvergentSite}

  We define the category $\AN^{\dagger}(\calV)$ of overconvergent varieties to be the localization of $\AN(\calV)$ by strict neighborhoods (which is possible by \cite{leStum:site}*{Proposition 1.3.6}): the objects of $\AN^{\dagger}(\calV)$ are the same as those of $\AN(\calV)$ and a morphism $(X',V') \to (X,V)$ in $\AN^{\dagger}(\calV)$ is a pair of formal morphisms 
\[
(X',V') \leftarrow (X',W) \to (X,V),
\]
where $(X',W) \to (X',V')$ is a strict neighborhood. 

The functor $\AN(\calV) \to \AN^{\dagger}(\calV)$ induces the image topology on $\AN^{\dagger}(\calV)$ (defined in \ref{p:inducedTopology} to be the largest topology on $\AN^{\dagger}(\calV)$ such that the map from $\AN^{\dagger}(\calV)$ is continuous. By \cite{leStum:site}*{Proposition 1.4.1}, the image topology on $\AN^{\dagger}(\calV)$ is generated by the pretopology of collections $\{(X,V_i) \to (X,V)\}$ with $\bigcup V_i$ an open covering of a neighborhood of $]X[_V$ in $V$ and $]X[_V = \bigcup \, ]X[_{V_i}$. 

\end{definition}

\begin{remark}

 From now on any morphism $(X',V') \to (X,V)$ of overconvergent varieties will  denote a morphism in $\AN^{\dagger}(\calV)$. One can give a down to earth description of morphisms in $\AN^{\dagger}(\calV)$ \cite{leStum:site}*{1.3.9}: to give a morphism $(X', V') \to (X,V)$, it suffices to give a neighborhood $W'$ of $]X'[_{V'}$ in $V'$ and a pair $f\colon X' \to X, u\colon W' \to V$ of morphisms which are \emph{geometrically pointwise compatible}, i.e., such that $u$ induces a map on tubes and the outer square of the diagram
\[
\xymatrix{
		W' 	\ar[r]^{u} 	&	V  		\\
	]X'[_{W'} 	\ar[r]^{]f[_u}
		\ar@{}[u]|{\bigcup |} \ar[d] & 		]X[_{V} 
		\ar@{}[u]|{\bigcup |} \ar[d]		\\
		X' 	\ar[r]^{f}		&	X 		
	}
\]
commutes (and continues to do so after any base change by any isometric extension $K'$ of $K$). 

\end{remark}

\begin{definition}
 
For any presheaf $T \in \widehat{\AN^{\dagger}(\calV)}$, we define $\AN^{\dagger}(T)$ to be the localized category $\AN^{\dagger}(\calV)_{/T}$ whose objects are morphisms $h_{(X,V)} \to T$ (where $h_{(X,V)}$ is the presheaf associated to $(X,V)$) and morphisms are morphisms $(X',V') \to (X,V)$ which induce a commutative diagram
\[
\xymatrix{
h_{(X',V')}\ar[rr]\ar[rd] && h_{(X,V)}\ar[ld]\\
&T&
}.
\]
We may endow $\AN^{\dagger}(T)$ with the induced topology (see \ref{p:inducedTopology}), i.e., the smallest topology making continuous the projection functor $\AN^{\dagger}(T) \to \AN^{\dagger}(\calV)$ \cite{leStum:site}*{Definition 1.4.7}; concretely, the covering condition is the same as in \ref{D:overconvergentSite}. When $T = h_{(X,V)}$ we denote $\AN^{\dagger}(T)$ by $\AN^{\dagger}(X,V)$. Since the projection $\AN^{\dagger} T \to \AN^{\dagger} \calV$ is a fibered category, the projection is also cocontinuous with respect to the induced topology.  Finally, an algebraic space $X$ over $k$ defines a presheaf $(X',V') \mapsto \Hom(X',X)$, and we denote the resulting site by $\AN^{\dagger}(X)$.

\end{definition}

There will be no confusion in writing $(X,V)$ for an object of $\AN^{\dagger}(T)$.\\

 We use subscripts to denote topoi and continue the above naming conventions  --  i.e., we denote the category of sheaves of sets on $\AN^{\dagger}(T)$ (resp. $\AN^{\dagger}(X,V), \AN^{\dagger}(X)$) by $T_{\AN^{\dagger}}$ (resp. $(X,V)_{\AN^{\dagger}}, X_{\AN^{\dagger}}$). Any morphism $f\colon T' \to T$ of presheaves on $\AN^{\dagger}(\calV)$ induces a morphism $f_{\AN^{\dagger}}\colon T'_{\AN^{\dagger}} \to T_{\AN^{\dagger}}$ of topoi. In the case of the important example of a morphism $(f,u)\colon (X',V') \to (X,V)$ of overconvergent varieties, we denote the induced morphism of topoi by $(u^*_{\AN^{\dagger}}, u_{\AN^{\dagger}*})$.

For an analytic space $V$ we denote by $\Open V$ the category of open subsets of $V$ and by $V_{\an}$ the associated topos of sheaves of sets on $\Open V$.  Recall that for an analytic variety $(X,V)$, the topology on the tube $]X[_V$ is induced by the inclusion $i_X\colon ]X[_V \, \hookrightarrow V$.

\begin{definition}[\cite{leStum:site}*{Corollary 2.1.3}]
\label{D:sheafRealization}

Let $(X,V)$ be an overconvergent variety. Then there is a  morphism of sites 
\[
\varphi_{X,V}\colon  \AN^{\dagger}(X,V) \to \Open \,  ]X[_V.
\]
The notation as usual is in the `direction' of the induced morphism of topoi and in particular backward; it is associated to the functor $\Open \,  ]X[_V \to \AN^{\dagger}(X,V)$ given by $U = W\cap\, ]X[_V\, 
 \mapsto \, (X,W)$ (and is independent of the choice of $W$ up to strict neighborhoods). This induces a morphism of topoi 
\[
(\varphi^{-1}_{X,V}, \varphi_{X,V*}) \colon (X,V)_{\AN^{\dagger}} \to (]X[_V)_{\an}.
\]

\end{definition}

\begin{definition}[\cite{leStum:site}*{2.1.7}]
 Let $(X,V) \in \AN^{\dagger}(T)$ be an overconvergent variety over $T$ and let $F \in T_{\AN^{\dagger}}$ be a sheaf on $\AN^{\dagger}(T)$. We define the \defi{realization} $F_{X,V}$ of $F$ on $]X[_V$ to be $\varphi_{(X,V)*}(F|_{(X,V)_{\AN^{\dagger}}})$, where $F|_{(X,V)_{\AN^{\dagger}}}$ is the restriction of $F$ to $\AN^{\dagger}(X,V)$.

\end{definition}

We can describe the category $T_{\AN^{\dagger}}$ in terms of realizations in a manner similar to sheaves on the crystalline or lisse-\'etale sites.

\begin{proposition}[\cite{leStum:site}, Proposition 2.1.8] 
\label{P:sheafReal}

Let  $T$ be a presheaf on $\AN^{\dagger}(\calV)$. Then the category $T_{\AN^{\dagger}}$ is equivalent to the following category:
\begin{enumerate}

\item An object is a collection of sheaves $F_{X,V}$ on $]X[_V$ indexed by $(X, V) \in \AN^\dagger(T)$ and, for each $(f, u) \colon (X', V') \to (X, V)$, a morphism $\phi_{f,u} : ]f[_u^{-1}  F_{X,V} \to  F_{X', V'}$, such that as $(f,u)$ varies, the maps $\phi_{f,u}$ satisfy the usual compatibility condition.

\item A morphism is a collection of morphisms $F_{X,V} \to G_{X,V}$ compatible with the morphisms $\phi_{f,u}$.

\end{enumerate}

\end{proposition}

To obtain a richer theory we endow our topoi with sheaves of rings and study the resulting theory of modules.

\begin{definition}[\cite{leStum:site}, Definition 2.3.4]

Define the \defi{sheaf of overconvergent functions} on $\AN^{\dagger}(\calV)$ to be the presheaf of rings 
    \[
    \mathcal O_{\mathcal V}^\dagger \colon (X, V) \mapsto \Gamma(]X[_V, i_X^{-1}\mathcal
    O_V)
    \]
    where $i_X$ is the inclusion of $]X[_V$ into $V$; this is a sheaf by \cite{leStum:site}*{Corollary 2.3.3}. For  $T \in \widehat{\AN^{\dagger}(\calV)}$ a presheaf on $\AN^{\dagger}(\calV)$, define $\calO^{\dagger}_{T}$ to be the restriction of $\calO^{\dagger}_{\calV}$ to $\AN^{\dagger}(T)$.

\end{definition}  

  We follow our naming conventions above, for instance denoting by $\calO^{\dagger}_{(X,V)}$ the restriction of $\calO^{\dagger}_{\calV}$ to $\AN(X,V)$.

\begin{remark}
\label{R:ringedRealization}

By \cite{leStum:site}*{Proposition  2.3.5, (i)}, the morphism of topoi of Definition \ref{D:sheafRealization} can be promoted to a morphism of ringed sites 
\[
    (\varphi_{X,V}^*, \varphi_{X,V*}) \colon (\textrm{AN}^\dagger(X,V), \mathcal O_{(X,V)}^\dagger) \to (]X[_V, i_X^{-1}\mathcal O_V).
\]
In particular, for $(X,V) \in \AN^{\dagger} T$ and $M \in \calO^{\dagger}_T$, the realization $M_{X,V}$ is an $i^{-1}_{X}\calO_V$-module. For any morphism $(f,u)\colon (X',V') \to (X,V)$ in $\AN^{\dagger}(T)$, one has a map
\[
 (]f[_{u}^\dagger, ]f[_{u*}) \colon (]X'[_{V'}, i_{X',V'}^{-1} \mathcal O_{V'}) \to (]X[_V, i_{X,V}^{-1} \mathcal O_V)
\]
of ringed sites, and functoriality gives transition maps
\[
    \phi^{\dagger}_{f,u}\colon ]f[_{u}^{\dagger}M_{X,V} \to M_{X',V'}
\]
which satisfy the usual cocycle compatibilities.

\end{remark}

 We can promote the description of $T_{\AN^{\dagger}}$ in Proposition \ref{P:sheafReal} to descriptions of the categories $\Mod \calO^{\dagger}_{T}$ of $\calO^{\dagger}_{T}$-modules, $\QCoh \calO^{\dagger}_{T}$ of quasi-coherent $\calO^{\dagger}_{T}$-modules (i.e., modules which locally have a presentation), and $\Mod_{\fp} \calO^{\dagger}_{T}$ of locally finitely presented $\calO^{\dagger}_{T}$-modules.

\begin{proposition}[\cite{leStum:site}, Proposition 2.3.6] 
\label{P:moduleReal}

Let  $T$ be a presheaf on $\AN^{\dagger}(\calV)$. Then the category $\Mod \calO^{\dagger}_T$ (resp. $\QCoh \calO^{\dagger}_T$,  $\Mod_{\fp} \calO^{\dagger}_T$) is equivalent to the following category:
\begin{enumerate}

\item An object is a collection of sheaves $M_{X,V} \in \Mod i^{-1}_X \calO_V$ (resp. $\QCoh i^{-1}_X\calO_V$,  $\Coh i^{-1}_X\calO_V$) on $]X[_V$ indexed by $(X, V) \in \AN^\dagger(T)$ and, for each $(f, u) \colon (X', V') \to (X, V)$, a morphism (resp. isomorphism) $\phi^{\dagger}_{f,u}\colon ]f[_u^{\dagger}M_{X,V} \to M_{X', V'}$, such that as $(f,u)$ varies, the maps $\phi^{\dagger}_{f,u}$ satisfy the usual compatibility condition.

\item A morphism is a collection of morphisms $M_{X,V} \to M'_{X,V}$ compatible with the morphisms $\phi_{f,u}^{\dagger}$.

\end{enumerate}

\end{proposition}

\begin{definition}[\cite{leStum:site}, Definition 2.3.7]

    Define the \defi{category of overconvergent crystals on $T$}, denoted $\Cris^{\dagger} T$, to be the full subcategory of $\Mod \calO^{\dagger}_T$ such that the transition maps $\phi_{f,u}^{\dagger}$ are isomorphisms.

\end{definition}  

\begin{example}

The sheaf $\calO^{\dagger}_{T}$ is a crystal, and in fact $\QCoh \calO^{\dagger}_T \subset \Cris^{\dagger} T$.

\end{example}

\begin{remark}
\label{R:crystalEquivalence}

It follows immediately from the definition of the pair $(\varphi_{X,V}^*, \varphi_{X,V*})$ of functors that $\varphi_{X,V*}$ of a  $\calO^{\dagger}_{(X,V)}$-module is a crystal, and that the adjunction $\varphi_{X,V}^* \varphi_{X,V*}E \to E$ is an isomorphism if $E$ is a crystal. If follows that the pair $\varphi_{X,V}^*$ and $\varphi_{X,V*}$ induce an equivalence of categories 
\[
\Cris^{\dagger} (X,V) \to \Mod i_X^{-1}\mathcal O_V;
\]
see \cite{leStum:site}*{Proposition 2.3.8} for more detail.

\end{remark}

\begin{remark}

An advantage of the use of sites and topoi is that the relative theory is simple. For instance, for a morphism $T' \to T$ of presheaves on $\AN^{\dagger}(\calV)$ the associated morphism of sites $\AN^{\dagger}(T') \to  \AN^{\dagger}(T)$ is isomorphic to the projection morphism associated to the localization $\AN^{\dagger}(T)_{/T'} \to \AN^{\dagger}(T)$ (and in particular one gets for free an exact left adjoint $u_!$ to the pullback functor $u^*\colon \Ab  (T_{\AN^{\dagger}}) \to \Ab (T'_{\AN^{\dagger}})$; see \ref{p:abelianExtensionZero}). 

\end{remark}

One minor subtlety is the choice of an overconvergent variety as a base.

\begin{definition}
\label{D:overconvergentBase}

Let $(C,O) \in \AN^{\dagger}(\calV)$ be an overconvergent variety and let $T \to C$ be a morphism from a presheaf on $\Sch_k$ to $C$. Then $T$ defines a presheaf on $\AN^{\dagger}(C,O)$ which sends $(X,V) \to (C,O)$ to $\Hom_C(X,T)$, which we denote by $T/O$. We denote the associated site by $\AN^{\dagger}(T/O)$, and when $(C,O) = (S_k,S)$ for some formal $\calV$-scheme $S$ we write instead $\AN^{\dagger}(T/S)$.

\end{definition}

The minor subtlety is that there is no morphism $T \to h_{(C,O)}$ of presheaves on $\AN^{\dagger}(\calV)$. A key construction is the following.

\begin{definition}[\cite{leStum:site}*{Paragraph after Corollary 1.4.15}]
\label{D:imagePresheaf}

Let $(X,V) \to (C,O) \in \AN^{\dagger}(\calV)$ be a morphism of overconvergent varieties. We denote by $X_V/O$ the image presheaf of the morphism $(X,V) \to X/O$, considered as a morphism of presheaves. Explicitly, a morphism $(X',V') \to X/O$ lifts to a morphism $(X',V') \to X_V/O$ if and only if there exists a morphism $(X',V') \to (X,V)$ over $X/O$, and in particular different lifts $(X',V') \to (X,V)$ give rise to the same morphism $(X',V') \to X_V/O$. When $(C,O) = (\Spec k, \MM(K))$, we may write $X_V$ instead $X_V/\MM(K)$.

\end{definition}

Many theorems will require the following extra assumption of \cite{leStum:site}*{Definition 1.5.10}. Recall that a morphism of formal schemes $P' \to P$ is said to be proper at a subscheme $X \subset P'_k$ if, for every component $Y$ of $\overline{X}$, the map $Y \to P_k$ is proper (see \cite{leStum:site}*{Definition 1.1.5}).

\begin{definition}
\label{D:realization}

Let $(C,O) \in \AN^{\dagger}(\calV)$ be an overconvergent variety and let $f\colon X \to C$ be a morphism of $k$-schemes. We say that a formal morphism $(f,u)\colon (X,V) \to (C,O)$, written as 
\[
\xymatrix{
X \ar@{^(->}[r] \ar[d]^f & P  \ar[d]^v &   V\ar[d]^u \ar[l]\\
C \ar@{^(->}[r] & Q  & O\ar[l]
},
\] 
is a \defi{geometric realization} of $f$ if $v$ is proper at $X$, $v$ is smooth in a neighborhood of $X$, and $V$ is a neighborhood of $]X[_{P_K \times_{Q_K} O}$ in $P_K \times_{Q_K} O$. We say that $f$ is \defi{realizable} if there exists a geometric realization of $f$.

\end{definition}

\begin{example}

Let $Q$ be a formal scheme and let $C$ be a closed subscheme of $Q$. Then any projective morphism $X \to C$ is realizable.

\end{example}

We need a final refinement to $\AN^{\dagger}(\calV)$.

\begin{definition}
\label{D:good}

We say that an overconvergent variety $(X,V)$ is \defi{good} if there is a good neighborhood $V'$ of $]X[_V$ in $V$ (i.e., every point of $]X[_V$ has an affinoid neighborhood in $V$). We say that a formal scheme $S$ is good if the overconvergent variety $(S_k,S_K)$ is good. We define the \defi{good overconvergent site} $\AN^{\dagger}_{\g}(T)$ to be the full subcategory of $\AN^{\dagger}(T)$ consisting of good overconvergent varieties. Given a presheaf $T \in \AN^{\dagger}(\calV)$, we denote by $T_{\g}$ the restriction of $T$ to $\AN^{\dagger}_{\g}(\calV)$.

\end{definition}

Note that localization commutes with passage to good variants of our sites (e.g., there is an isomorphism $\AN^{\dagger}_{\g}(\calV)_{/T_{\g}} \cong \AN^{\dagger}_{\g}(T)$). When making further definitions we will often omit the generalization to $\AN^{\dagger}_{\g}$ when it is clear. 
\\

The following proposition will allow us to deduce facts about $\Mod_{\fp} \calO^{\dagger}_{X_g}$ from results about $(X,V)$ and $X_V$.

\begin{proposition}
\label{P:coverings}

Let $(C,O) \in \AN^{\dagger}_{\g}(\calV)$ be a good overconvergent variety and let $(X,V) \to (C,O)$ be a geometric realization of a morphism $X \to C$ of schemes. Then the following are true:
\begin{itemize}
\item [(i)] The map $(X,V)_{\g} \to (X/O)_{\g}$ is a covering in $\AN^{\dagger}_{\g}(\calV)$.
\item [(ii)] There is an equivalence of topoi $(X_V/O)_{\AN^{\dagger}_g} \cong (X/O)_{\AN^{\dagger}_g}$. 
\item [(iii)] The natural pullback map $\Cris_{\g}^{\dagger} X/O \to \Cris_{\g}^{\dagger} X_{V}/O$ is an equivalence of categories.
\item [(iv)] Suppose that $(X,V)$ is good. Then the natural map $\Cris^{\dagger} X_V/O \to \Cris_{\g}^{\dagger} X_V/O$ is an equivalence of categories.
\end{itemize}

\end{proposition}

\begin{proof}

The first two claims are \cite{leStum:site}*{1.5.14,  1.5.15}, the third follows from the second, and the last is clear.

\end{proof}

In particular, the natural map $\Mod_{\fp} \calO^{\dagger}_{X_{\g}} \to \Mod_{\fp} \calO^{\dagger}_{(X_V)_{\g}} \cong \Mod_{\fp} \calO^{\dagger}_{X_V}$ is an equivalence of categories.

\section{Calculus on the overconvergent site and comparison with the classical theory}
\label{S:overconvergentCalculus}

Here we compare constructions on the overconvergent site and on the ringed spaces $(]X[_V, i_X^{-1}\calO_V)$ to the classical constructions of rigid cohomlogy (exposited for example in \cite{leStum:rigidBook}), introducing along the way the variants of `infinitesimal calculus' useful in the following. 
\\

Let $V$ be an analytic variety. Recall (see \cite{leStum:site}*{Section 4.2}) that $V$ has a Grothendieck topology (generated by affinoid subdomains) which is finer than its usual topology; we refer to this as `the $G$-topology' on $V$ and write $V_G$ when we consider $V$ with its $G$-topology. The natural morphism $\pi\colon V_G \to V$ (induced by the morphism $\id\colon V \to V_G$ on underlying sets) is a morphism of ringed sites. When $V$ is good, the functor $F \mapsto F_G := \pi^*F$ is fully faithful and induces an equivalence of categories 
\[
\Coh \calO_V \cong \Coh \calO_{V_G}.
\]
Indeed, for an admissible $W \in \tau_V$, $\pi^*F(W) = \varinjlim_{W \subset W'}F(W')$, where the limit is taken over all open neighborhoods $W'$ of $W$. The unit $\id \to \pi_*\pi^*F$ of adjunction is then visibly an isomorphism, so by lemma \ref{L:unitFull}, we conclude that $\pi^*$ is an isomorphism.

Recall also that the set $V_0$ of rigid points of $V$ has the structure of a rigid analytic variety such that the inclusion $V_0 \hookrightarrow V$ induces an equivalence $(V_0,\calO_{V_0}) \cong (V_G,\calO_{V_G})$ of ringed topoi, in particular inducing equivalences 
\[
\Mod \calO_{V_0} \cong \Mod \calO_{V_G}
\]
and 
\[
\Coh \calO_{V_0} \cong \Coh \calO_{V_G} \cong \Coh \calO_{V}. 
\]
We denote by $\pi_0$ the composition $\widetilde{V_0} \cong \widetilde{V_G} \to \widetilde{V}$, and for a bounded below complex of abelian sheaves $E_0 \in  \mathbb{D}_+(\widetilde{V_0})$ define $E_0^{\an}$ to be $\mathbb{R}\pi_0E$. When $V$ is good and $E_0$ is coherent there is an isomorphism $E_0^{\an} \cong \pi_0E_0$ (this follows from \cite{Berkovich:nonArchEtaleCoh}*{1.3.6 (ii)}). Moreover, suppose that $(X,V)$ is a good overconvergent variety, $]\overline{X}[_V = V$, and $E_0$ is a coherent $j_{X_0}^{\dagger}\calO_{V_0}$-module (see Definition \ref{D:jDagger} below). Then by \cite{leStum:site}*{Proposition 3.4.3 (3)},  $E_0^{\an} \cong \pi_0E$.

Now let $(X,V)$ be a good overconvergent variety. We studied above (Proposition \ref{P:moduleReal}) the ringed site $(]X[_V, i_X^{-1}\calO_V)$. To study the analogue in the classical rigid theory and to compare the two we first make the following definitions.

\begin{definition}
\label{D:jDagger}

Let $(X,V)$ be a good overconvergent variety and assume that the inclusion $i_X\colon ]X[_V\, \hookrightarrow V$ is closed (which we can do since $(X,V) \cong (X,]\overline{X}[_V)$ in $\AN^{\dagger}(\calV)$, where $\overline{X}$ is the closure of $X$ in $P$). Let $]X[_{V_0}$ be the underlying rigid space $(]X[_{V})_0$ of $]X[_{V}$; alternatively,  $]X[_{V_0}$ is isomorphic to the rigid analytic tube, i.e., the preimage of $X$ with respect to the composition $V_0 \to (P_K)_0 \to P_k$, where $(P_K)_0$ is the Raynaud generic fiber of $P$.

Denote by $i_{X_0}\colon ]X[_{V_0} \hookrightarrow V_0$ the corresponding inclusion of rigid analytic spaces and let $F \in \widetilde{V}$ (resp. $F_0 \in \widetilde{V_0}$). We define functors $j_X^{\dagger}$ \cite{leStum:site}*{Proposition 2.2.12} and $j_{X_0}^{\dagger}$ \cite{leStum:rigidBook}*{Proposition 5.1.2} by 
\[
j_{X}^{\dagger} F = i_{X*}i_X^{-1} F  
\]
and 
\[
j_{X_0}^{\dagger} F_0 = \varinjlim j'_{0*}j^{'-1}_0 F_ 0
\]
where the limit runs over all strict neighborhoods $V_0'$ of $]X[_{V_0}$ in $V_0$ (recall from \cite{leStum:rigidBook}*{Definition 3.1.1} that a strict neighborhood of $]X[_{V_0}$ in $V_0$ is an admissible open subset $V_0'$ containing $]X[_{V_0}$ such that the covering $\{V_0', V_0 - ]X[_{V_0}\}$ is an admissible covering of $V_0$).

\end{definition}

\begin{proposition}
\label{P:daggerYoga}

With the notation of Definition \ref{D:jDagger}, the following are true.
\begin{itemize}
  
\item [(i)] There is a natural isomorphism 
\[
j_{X}^{\dagger} F = \varinjlim j'_{*}j^{'-1} F
\]
where the limit runs over all immersions of neighborhoods $V'$ of $]X[_{V}$ in $V$.

\item [(ii)] The functors $i_X^{-1}$ and $i_{X*}$ induce an equivalence of categories 
\[
\QCoh j_X^{\dagger}\calO_V \to \QCoh i_X^{-1} \calO_V
\]
which restricts to give an equivalence on coherent sheaves.

\item [(iii)] The functors 
\[
\Coh \calO_{V'} \to \Coh i_X^{-1} \calO_V \to \Coh j_X^{\dagger} \calO_V,
\]
\[
 E \mapsto i_{X,V'}^{-1}E \mapsto i_{X*}i_{X,V'}^{-1}E,
\]
where $V'$ ranges over neighborhoods of $]X[_V$ in $V$ and $i_{X,V'}$ denotes the inclusion $]X[_V\, \hookrightarrow V'$, induce equivalences of categories 
\[
\varinjlim \Coh \calO_{V'} \cong \Coh i_X^{-1} \calO_V \cong \Coh j_X^{\dagger} \calO_V.
\]

\item [(iii')] The functors 
\[
\Coh \calO_{V'} \to \Coh j_{X_0}^{\dagger} \calO_{V_0}, \, E \mapsto i_{X_0*}i_{X_0,V'}^{-1}E,
\]
where $V'$ ranges over strict neighborhoods of $]X[_{V_0}$ in $V_0$ and $i_{X_0,V'}$ denotes the inclusion $]X[_{V_0} \hookrightarrow V'$, induce an equivalence of categories 
\[
\varinjlim \Coh \calO_{V'} \cong \Coh j_{X_0}^{\dagger} \calO_{V_0}.
\]

\item [(iv)] The map $E \mapsto E^{\an}$ induces an equivalence of categories 
\[
\Coh j_{X_0}^{\dagger} \calO_{V_0} \cong \Coh j_X^{\dagger} \calO_V.
\]

\end{itemize}

\end{proposition}

In particular, $\Mod_{\fp} \calO^{\dagger}_{X,V}$ is equivalent to $\Coh j^{\dagger}_{X_0} \calO_{V_0}$.

\begin{proof}

Claim (i) is \cite{leStum:site}*{2.2.12}. 

For (ii), it suffices to check that the unit $\id \to i_{X*}i_X^{*}$ and counit $i_X^{*}i_{X*} \to \id$ of adjunction are isomorphisms (where $i_X^*$ is the composition of $i_X^{-1}$ and tensoring). By Remark \ref{R:geometricImmersions} the inclusion $i_X$ induces an immersion of topoi, and in particular the map $i_{X*}$ is fully faithful, so by Lemma \ref{L:unitFull} we conclude that the adjunction $i_X^{*}i_{X*} \to \id$ is an isomorphism. For the other direction, let $E \in \QCoh j_X^{\dagger}\calO_V$. We can check locally that the adjunction is an isomorphism, so we may assume that $E$ has a global presentation. Since $]X[_V$ is closed in $V$, $i_{X*}$ is exact \cite{Berkovich:nonArchEtaleCoh}*{4.3.2} (and $i^{-1}_X$ is always exact) so that the adjunction induces a diagram
\[
\xymatrix{
\bigoplus_I j_X^{\dagger}\calO_V   \ar[r] \ar[d] & 
  \bigoplus_J j_X^{\dagger}\calO_V \ar[r] \ar[d] & 
  E \ar[r] \ar[d] &
  0 \\
\bigoplus_I i_{X*}i_X^{-1}j_X^{\dagger}\calO_V  \ar[r]  & 
  \bigoplus_J i_{X*}i_X^{-1}j_X^{\dagger}\calO_V \ar[r]  & 
  i_{X*}i_X^{-1}E \ar[r] &
  0
}.
\]  
Thus to prove the claim it is thus enough to check that the adjunction $j_X^{\dagger} \calO_V \to i_{X*}i_X^{-1}j_X^{\dagger}\calO_V$ is an isomorphism, which is true since we can write this as $i_{X*}i_X^{-1}\calO_V \to i_{X*}i_X^{-1}i_{X*}i_X^{-1}\calO_V$, which is $i_{X*}$ applied to the adjunction $i_X^{-1}\calO_V \to i_X^{-1}i_{X*}i_X^{-1}\calO_V$ and thus an isomorphism (by the beginning of this paragraph).

Claim (iii) is \cite{leStum:site}*{Proposition 2.2.12} (and (ii)) and claim (iii') is \cite{leStum:rigidBook}*{Theorem 5.4.4}.

Finally, note that, from the explicit description of the functor $E \mapsto E^{\an}$, following diagram commutes
\[
\xymatrix{
\varinjlim \Coh \calO_{V_0'} \ar[r]\ar[d] & \Coh j_{X_0}^{\dagger} \calO_{V_0}\ar[d] \\
\varinjlim \Coh \calO_{V'} \ar[r] & \Coh j_X^{\dagger} \calO_V
}.
\]
Claim (iv) then follows from (iii) and (iii') together with \cite{leStum:site}*{Corollary 1.3.2} (which says that there is a cofinal system of neighborhoods $\{V'\}$ such that the system $\{V_0'\}$ is a cofinal system of strict neighborhoods) and the isomorphism 
\[
\Coh \calO_{V_0'} \cong \Coh \calO_{V'}.
\]

\end{proof}  

\begin{remark}
\label{R:jDaggerDescription}

A benefit of using Berkovich spaces instead of rigid analytic spaces is that the analogous construction $i_{X_0}^{-1}\calO_{V_0}$ in rigid geometry does not serve the same purpose, since the closed inclusion $i_{X_0}\colon ]X[_{V_0}\, \hookrightarrow V_0$ is also open and so $i_{X_0*}i_{X_0}^{-1} \calO_{V_0}$ is not isomorphic to $j_{X_0}^{\dagger}\calO_{V_0}$. If instead one lets $U$ denote the open complement of $]X[_{V_0} \subset V_0$ and then denotes by $i\colon Z \hookrightarrow \widetilde{V_0}$ the closed complement of $U \subset V_0$ in the sense of Section \ref{S:excisionTopoi}, then by \cite{leStum:rigidBook}*{5.1.12 (i)} the functor $j_{X_0}^{\dagger}$ is isomorphic to $i_*i^*$. This is a nice instance of the utility of the abstract notion of an immersion of topoi.

\end{remark}  

\subsection{Infinitesimal Calculus} 
We recall here several definitions from \cite{leStum:site}*{Section 2.4}.

Let $V,V' \to O$ be two morphisms of analytic spaces. Then by \cite{Berkovich:nonArchEtaleCoh}*{Proposition 1.4.1}, the fiber product $V\times_O V'$ exists -- when $V,V'$ and $O$ are affinoid spaces the fiber product is given by the Gelfand spectrum of the completed tensor product of their underlying algebras, and the global construction is given by glueing this construction. As usual the underlying topological space of $V\times_OV'$ is not the fiber product of their underlying topological spaces.

Let $V \to O$ be a morphism of analytic varieties. Then the diagonal morphism $\Delta\colon V \to V \times_O V$ is a $G$-locally closed immersion (see the comments after the proof of \cite{Berkovich:nonArchEtaleCoh}*{Proposition 1.4.1}). We define the relative sheaf of differentials $\Omega_{V/O}$ to be the conormal sheaf of $\Delta$. When $\Delta$ is a closed immersion defined by an ideal $I$, $\Omega_{V/O}$ is the restriction of $I/I^2$ to $V$; in general one can either define the conormal sheaf locally (and check that it glues) or argue that when $V$ is good, $\Delta$ factors as composition of a closed immersion $i\colon V \hookrightarrow U$ into an admissible open $U$, with $i$ defined by an ideal $J$, and define the conormal sheaf as the restriction of $J/J^2$.

Due to the use of completed tensor products, the sheaf of differentials is generally not isomorphic to the sheaf of Kahler differentials. It does however enjoy all of the usual properties; see \cite{Berkovich:nonArchEtaleCoh}*{3.3}.

\begin{definition}
\label{D:MIC}
 
Let $(X,V) \to (C,O)$ be a morphism of overconvergent varieties. Suppose that $V$ is good and that $i^{-1}_X\colon ]X[_V \hookrightarrow V$ is closed. We define the category $\MIC(X,V/O)$ of \defi{overconvergent modules with integrable connection} to be the category of pairs $(M,\nabla)$, where $M \in \Mod i_X^{-1}\calO_V$ and $\nabla \colon M \to M \otimes_{i^{-1}_X\calO_V}i_X^{-1}\Omega^1_{V/O}$ is an $i_C^{-1}\calO_O$-linear map satisfying the Leibniz rule and such that the induced map $\nabla \circ \nabla\colon M \to M \otimes_{i^{-1}_X\calO_V}i_X^{-1}\Omega^2_{V/O}$ is zero. Morphisms $(M,\nabla) \to (M', \nabla')$ are morphisms $M \to M'$ as $i_X^{-1}\calO_V$-modules which respect the connections (see \cite{leStum:site}*{Definition 2.4.5}). Similarly, we define a category $\MIC(X_0,V_0/O_0)$ of such pairs $(M_0,\nabla_0)$ with $M_0 \in \Coh j_{X_0}^{\dagger} \calO_{V_0}$ (see \cite{leStum:rigidBook}*{Definition 6.1.8}).

Let $(M,\nabla) \in \MIC(X,V/O)$. Then $\nabla$ extends to a complex 
\[
 M \to M \otimes_{i^{-1}_X\calO_V}i_X^{-1}\Omega^1_{V/O} \to M \otimes_{i^{-1}_X\calO_V}i_X^{-1}\Omega^2_{V/O} \to M \otimes_{i^{-1}_X\calO_V}i_X^{-1}\Omega^3_{V/O} \to ...
\]
of abelian sheaves, which we call the \defi{de Rham complex} of $(E,\nabla)$ and write as $M \otimes_{i^{-1}_X\calO_V}i_X^{-1}\Omega^{\bullet}_{V/O}$. We define the de Rham complex of  $(M_0,\nabla_0) \in \MIC(X_0,V_0/O_0)$ similarly.

\end{definition}

The bridge between crystals and modules with integrable connection is the notion of a stratification, which we now define.

\begin{definition}
\label{D:strat}

Let $(X,V) \to (C,O)$ be a morphism of overconvergent varieties. Set $V^2 = V \times_O V$ and denote by 
\[
p_1,p_2\colon (X,V^2) \to (X,V)
\]
the two projections. We define an \defi{overconvergent stratification} on an $i_X^{-1}\calO_V$-module $M$ to be an isomorphism 
\[
\epsilon\colon p_2^{\dagger}M \cong p_1^{\dagger}M
\]
of $i^{-1}_X\calO_{V^2}$-modules satisfying the evident cocycle condition on triple products (see for example \cite{BerthelotO:notesCrystalline}*{Definition 2.10}). We denote the category of such pairs $(M,\epsilon)$ by $\Strat^{\dagger} (X,V/O)$, where morphisms are morphisms of $i_X^{-1}\calO_V$-modules which respect the stratification. We define the rigid variant $\Strat^{\dagger}(X_0,V_0/O_0)$ analogously.

\end{definition}

We omit a discussion of the notion of more general (than overconvergent) stratifications.
\\

\begin{remark}

One can relate crystals and overconvergent stratifications as follows. Let $(X,V) \to (C,O)$ be a morphism of overconvergent varieties and suppose that $(X,V)$ is a good overconvergent variety. Let $E \in \Cris^{\dagger} X_V/O$ and consider the diagram 
\[
\xymatrix{
(X,V^2) \cong (X,V)\times_{X_V/O}(X,V) 
  \ar@<+.5ex>[r]^<>(.5){p_1} \ar@<-.5ex>[r]_<>(.5){p_2} &
(X,V) 
  \ar[r]^p &
X_V/O
}.
\]
Then the composition $\epsilon$ of the two isomorphisms
\[
\epsilon\colon p_2^{\dagger}E_{X,V} \cong E_{X,V^2} \cong p_1^{\dagger} E_{X,V}
\]
(which exist by applying the condition that $E$ is a crystal to the maps $p_i$) defines a stratification on $E_{X,V}$ and thus a functor 
\begin{equation}
\label{eq:crisToStrat}
\Cris^{\dagger}X_V/O  \to  \Strat^{\dagger} (X,V/O), 
\end{equation} 
given by $E \mapsto (E_{X,V}, \epsilon)$. On the other hand, a stratification on $E \in \Mod i_X^{-1} \calO_V$ defines descent data on the crystal $\varphi^{*}_{X,V} E$ with respect to the map $p\colon (X,V) \to X_V/O$; by definition the map $p$ is a surjection of presheaves and thus a covering (in the canonical topology). By descent theory, the map \ref{eq:crisToStrat} is an equivalence of categories (see \cite{leStum:site}*{2.5.3}). 
\\

\end{remark}

\begin{remark}

Here we relate the notion of a stratification and a module with connection. There is a map 
\[
\Strat^{\dagger} (X,V/O) \to \MIC (X,V/O)
\]
defined via the usual yoga of  `infinitesimal calculus', which we now recall. Let $V \hookrightarrow V^2 := V \times_O V$ be the diagonal morphism and denote by $V^{(n)}$ the $n^{\text{th}}$ infinitesimal neighborhood of the diagonal (when $V^{(0)} = V \hookrightarrow V^2$ is defined by an ideal $I$, and $V^{(n)} \hookrightarrow V^2$ is defined by the ideal $I^{n+1}$; in general one defines $V^{(n)}$ locally and glues). By definition the sequence
\[
0 \to \Omega^1_{V/O} \to \calO_V^{(1)} \to \calO_V \to 0
\]
is exact. We denote by $p_1^{(n)}$ and $p_2^{(n)}$ the two compositions 
\[
\xymatrix{
V^{(n)}
   \ar[r] &
V^2 
  \ar@<+.5ex>[r]^<>(.5){p_1} \ar@<-.5ex>[r]_<>(.5){p_2} &
V
}.
\]
Let $(M, \epsilon) \in \Strat^{\dagger} (X,V/O)$ be a module with an overconvergent stratification. Then $\epsilon$ restricts to give a compatible system $\{\epsilon^{(n)}\colon p_2^{(n)\dagger}M \cong p_1^{(n)\dagger}M\}$ of isomorphisms on $]X[_{V^{(n)}}$.

Denote by $\theta_i$ the natural map 
\[
\theta_i\colon M \to p_i^{(1)\dagger}M = M\otimes_{i_X^{-1}\calO_{V}} i_X^{-1}\calO_{V^{(1)}}
\] 
given by tensoring (noting that the underlying topological spaces of $V^{(i)}$ are the same). We define a connection $\nabla$ on $M$ by the formula
\[
\nabla = (\epsilon \circ \theta_2) - (\theta_1)\colon M \to p_1^{(1)\dagger}M = M\otimes_{i_X^{-1}\calO_{V}} i_X^{-1}\calO_{V^{(1)}}.
\]
The map $\nabla$ lands in $M\otimes_{i_X^{-1}\calO_{V}} i_X^{-1}\Omega^1_{V/O}$ by the description above of $\Omega^1_{V/O}$ together with the observation that since the two compositions
\[
\xymatrix{
V
  \ar[r] &
V^{(n)}
  \ar@<+.5ex>[r]^<>(.5){p_1} \ar@<-.5ex>[r]_<>(.5){p_2} &
V
}
\]
are equal, the composition 
\[
 M \to  M\otimes_{i_X^{-1}\calO_{V}} i_X^{-1}\calO_{V^{(1)}} \to M\otimes_{i_X^{-1}\calO_{V}} i_X^{-1}\calO_{V}
\]
is zero. Integrability of $\nabla$ follows from the cocycle condition.

\end{remark}

Next we mildly refine the notion of a connection.

\begin{definition}

Let $(M, \nabla) \in \MIC(X,V/O)$. As in \cite{leStum:site}*{Definition 2.4.6}, we say that $\nabla$ is \defi{overconvergent} if $M$ is coherent and $(M, \nabla)$ is in the image of the map
\[
\Strat^{\dagger} (X,V/O) \to \MIC (X,V/O)
\]
and we denote the category of overconvergent modules with integrable connection by $\MIC^{\dagger}(X,V/O)$. We define $\MIC^{\dagger}(X_0,V_0/O_0)$ similarly, where $E_0$ is a $j_{X_0}^{\dagger}\calO_{V_0}$-module and the connection is a map $\nabla \colon E_0 \to E_0 \otimes_{\calO_{V_0}}\Omega^1_{V_0/O_0}$. 

When $(C,O) = (S_k,S_K)$ for a formal scheme $S$ and $V = P_K$ for a formal embedding $X \hookrightarrow P$ of $X$ into some formal scheme over $S$, we denote $\MIC^{\dagger}(X_0,V_0/O_0)$ by $\Isoc^{\dagger}(X \subset \overline{X}/S)$; by \cite{leStum:rigidBook}*{Corollary 8.1.9} this category only depends on the closure $\overline{X}$ of $X$ in $P$ and is independent of the choice of $P$, which we therefore omitted from the notation.

\end{definition}

\begin{corollary}

The natural map $\MIC^{\dagger}(X_0,V_0/O_0) \to \MIC^{\dagger}(X,V/O)$ is an equivalence of categories.

\end{corollary}

\begin{proof}

It suffices to prove this for coherent modules with overconvergent stratification, which follows from Proposition \ref{P:daggerYoga} (iii) and (iv).

\end{proof}  

\begin{remark}

The composition 
\[
\Cris^{\dagger}X_V/O \cong \Strat^{\dagger} (X,V/O) \to \MIC (X,V/O)
\] 
induces an equivalence of categories
\[
\Mod^{\dagger}_{\fp}(X_V/O) \cong \MIC^{\dagger}(X,V/O);
\]
see \cite{leStum:site}*{remark after 2.4.5}.

\end{remark}

The following theorem of le Stum ties this discussion together with Proposition \ref{P:coverings} to give an intrinsic characterization of isocrystals via the good overconvergent site and in particular gives a new proof of the independence of $\Isoc^{\dagger}(X \subset \overline{X})$ from the choice of compactification $\overline{X}$.

\begin{theorem}[\cite{leStum:site}, Corollary 2.5.11]
\label{T:coefficients}

Let $S$ be a formal $\calV$-scheme and let $X/S_k$ be a realizable algebraic variety. Then there is an equivalence of categories $\Mod_{\fp}(\calO^{\dagger}_{X_{\g}/S}) \cong \Isoc^{\dagger}(X \subset \overline{X}/S)$.

\end{theorem}  

Le Stum proves a similar result for cohomology, which we recall below.

\begin{definition}[\cite{leStum:site}*{Definition 3.5.1}]

Let $(C,O)$ be an overconvergent variety, let $f\colon X' \to X$ be a morphism of schemes over $C$. Then $f$ induces a morphism of topoi $f_{\AN^{\dagger}_{\g}} \colon X'/O_{\AN^{\dagger}_{\g}} \to X/O_{\AN^{\dagger}_{\g}}$. For $F \in (X'/O)_{\AN^{\dagger}_{\g}}$ be a sheaf of abelian groups (or more generally any bounded below complex of abelian sheaves) we define the \defi{relative rigid cohomology} of $F$ to be $\mathbb{R}f_{\AN^{\dagger}_{\g}*}F$. 

When $(C,O) = (\Spec k, \calM(K))$ and $X = \Spec k$, for an integer $i \geq 0$ we define the \defi{absolute rigid cohomology} of $F$ to be the $K$-vector space $H^i(\AN^{\dagger}_{\g} X', F) :=  (\mathbb{R}^if_{\AN^{\dagger}_{\g}*}F)_{(\Spec k, \calM(K))}$; since the realization functor is exact this is isomorphic to the $i^{\Th}$ derived functor of the global sections functor. When $F = \calO^{\dagger}_{X'_{\g}}$, we write $H^i(\AN^{\dagger}_{\g} X') := H^i(\AN^{\dagger}_{\g} X', \calO^{\dagger}_{X'_{\g}})$.

\end{definition}

\begin{remark}

  The functor $F \mapsto F_{\g}$ is exact on abelian sheaves (since goodness of $(X,V)$ is local on $V$), so when computing cohomology we can derive either of the functors $f_{\AN^{\dagger}*}$ or $f_{\AN^{\dagger}_{\g}*}$.

\end{remark}

Now we explain how to compare the cohomology on the overconvergent site to classical rigid cohomology. Let $j_{X,V}\colon (X,V) \to X/O$ be an overconvergent variety over $X/O$ and let $E \in \Cris_{\g}^{\dagger} X/O$ be a crystal. Then the adjunction
\[
E \to j_{X,V*}j_{X,V}^*E \cong j_{X,V*}\varphi_{X,V}^{*}E_{X,V}
\]
(where the second map is an isomorphism by Remark \ref{R:crystalEquivalence}) induces by \cite{leStum:site}*{Proposition 3.3.10 (ii)} a map 
\[
E \to \mathbb{R}j_{X,V*}\left(\varphi_{X,V}^* \left(E_{X,V}\otimes_{i_X^{-1}\calO_V} i_X^{-1}\Omega^{\bullet}_{V/O}\right)\right)
\]
of complexes of $\calO^{\dagger}_{(X/O)_{\g}}$-modules. In the following nice situation this map is a quasi-isomorphism and we can thus compute the cohomology of $E$ via the cohomology of the de Rham complex $E_{X,V}\otimes_{i_X^{-1}\calO_V} i_X^{-1}\Omega^{\bullet}_{V/O}$.

\begin{theorem}
\label{T:deRham}

Let $(C,O)$ be an overconvergent variety and suppose that $(X,V)$ is a geometric realization of the morphism $X \to C$, and denote by $p_{\AN^{\dagger}_{\g}}$ the morphism of topoi $p_{\AN^{\dagger}_{\g}} \colon (X/O)_{\AN^{\dagger}_{\g}} \to (C,O)_{\AN^{\dagger}_{\g}}$. 
Then the following are true.

\begin{itemize}
\item [(i)] The augmentation 
\[
E \to \mathbb{R}j_{X,V*}\left(\varphi_{X,V}^* \left(E_{X,V}\otimes_{i_X^{-1}\calO_V} i_X^{-1}\Omega^{\bullet}_{V/O}\right)\right)
\]
is an isomorphism.

\item [(ii)] The natural map  
\[
\left(\mathbb{R}p_{\AN^{\dagger}_{\g}*} E \right)_{C,O} \to \mathbb{R}p_{]X[_{V}*}\left(E_{X,V}\otimes_{i_X^{-1}\calO_V} i_X^{-1}\Omega^{\bullet}_{V/O}\right)
\]
(induced by part (i)) is a quasi-isomorphism.

\end{itemize}

\end{theorem}

Of course, one can compute any other realization $\left(\mathbb{R}p_{\AN^{\dagger}_{\g}*} E \right)_{C',O'}$ of relative rigid cohomology from (ii) by base change \cite{leStum:site}*{Corollary 3.5.7}.

\begin{proof}

Claim (i) is \cite{leStum:site}*{Proposition 3.5.4} and claim (ii) is \cite{leStum:site}*{Theorem 3.5.3} (which follows from (i) by \cite{leStum:site}*{Proposition 3.3.9}).

\end{proof}

One can compare this with the classical notions of rigid cohomology, which we now recall.

\begin{definition}[\cite{leStum:rigidBook}*{Definition 8.2.5}]

 Let $S$ be a formal $\calV$-scheme, let $f\colon X \to S_k$ be a morphism of algebraic varieties, let $X \hookrightarrow P$ be a formal embedding over $S$ and denote by $g$ the map $]\overline{X}[_{(P_K)_0} \to (S_K)_0$. Let $E_0 \in \Isoc^{\dagger}(X \subset \overline{X}/S) := \MIC^{\dagger}(X_0,(P_K)_0/(S_K)_0)$. We define the classical rigid cohomology $\mathbb{R}f_{\rig}E_0$ of $E_0$ to be the higher direct image $\mathbb{R}g_*(E_0 \otimes \Omega_{]\overline{X}[_{(P_K)_0}/(S_K)_0}^{\bullet})$ of the de Rham complex associated to $(E_0, \nabla)$ (considered as a complex of abelian sheaves). When $S = \Spf \calV$, we call this the absolute rigid cohomology and denote its $i^{\Th}$ homology by $H^i_{\rig}(X,E_0)$.

\end{definition}

Actually, rigid cohomology is independent of the choice of $P$ and $\overline{X}$ \cite{leStum:rigidBook}*{Proposition 8.2.1}, which we thus do not mention in the following theorem. When no choice of $P$ exists one can define $\Isoc^{\dagger}(X)$ and rigid cohomology by cohomological descent \cite{ChiarellottoT:etaleRigidDescent}.

\begin{theorem}[\cite{leStum:site}, Proposition 3.5.8]
\label{T:cohomology}

 Let $S$ be a formal $\calV$-scheme such that $(S_k,S_K)$ is a good overconvergent variety and let $f\colon X \to S_k$ be a morphism of algebraic varieties. Let $(X,P)$ be a geometric realization of $X \to S_k$ and denote by $\overline{X}$ the closure of $X$ in $P$. Then for any $E \in \Mod^{\dagger}_{\fp,\g} (X/S)$ and $E_0 \in \Isoc^{\dagger}(X \subset \overline{X}/S)$ such that $E_{X,P} \cong i_X^{-1}E_0^{\an}$, there is a natural map (in the derived category)
	\[
	  i_{S_k}^{-1}(\mathbb{R}f_{\rig}E_0)^{\an} \to (\mathbb{R}f_{\AN^{\dagger}_{\g} *}E)_{(S_k,S_K)}
	\]
which is a quasi-isomorphism.

\end{theorem}

The natural map is constructed as follows. Denote by $V$ the tube $]\overline{X}[_{P_K}$, by $O$ the analytic space $S_K$, and by $u$ the map $V \to O$. There is a natural map 
\[
\left(\mathbb{R}f_{\rig}E_0\right)^{\an} = 
\left(\mathbb{R}u_{0*}\left(E_0 \otimes_{\calO_{O_0}} \Omega^{\bullet}_{V_0/O_0}\right)\right)^{\an} \to  
\mathbb{R}u_*\left(E_0 \otimes_{\calO_{O_0}} \Omega^{\bullet}_{V_0/O_0}\right)^{\an}.
\]
Since $V$ is smooth in a neighborhood of the tube $]X[_{P_K}$, $\Omega^{\bullet}_{V_0/O_0}$ is locally free in such a neighborhood. Thus the tensor product $E_0 \otimes_{\calO_{O_0}} \Omega^{\bullet}_{V_0/O_0}$ has coherent terms and analytifies to $E' \otimes_{\calO_{O}} \Omega^{\bullet}_{V/O}$, where $i^{-1}_X E' \cong E_{(S_k,S_K)}$. Furthermore, since $i_X^{-1}$ and $i_{S_k}^{-1}$ are exact, there are isomorphisms
\[
i^{-1}_{S_k}\mathbb{R}u_*\left(E' \otimes_{\calO_{O}} \Omega^{\bullet}_{V/O}\right) 
\cong 
\mathbb{R}_{]f[_*}i^{-1}_X \left(E' \otimes_{\calO_{O}} \Omega^{\bullet}_{V/O}\right)  
\cong 
\mathbb{R}_{]f[_*} \left( E_{(S_k,S_K)} \otimes_{i^{-1}_X\calO_{O}} i^{-1}_X \Omega^{\bullet}_{V/O}\right)
\]
By Theorem \ref{T:deRham} (ii), the last term is isomorphic to $\left(\mathbb{R}f_{\AN^{\dagger}_{\g}*}E\right)_{(S_k,S_K)}$. Applying $i^{-1}_{S_k}$ and composing these isomorphisms gives the natural map.
\\

We end by stating a corollary of the comparison theorem.

\begin{theorem}
\label{T:cohomologyCorollary}

Let $X$ be an algebraic variety over $k$. Let $(X,P)$ be a geometric realization of $X \to S_k$ and denote by $\overline{X}$ the closure of $X$ in $P$. Then for any $E \in \Mod^{\dagger}_{\fp} (X_{\g})$ and $E_0 \in \Isoc^{\dagger}(X \subset \overline{X}/k)$ such that $E_{X,P} \cong i_X^{-1}E_0^{\an}$, there is a natural map
	\[
	  H^i_{\rig}(X,E_0) \to H^i(\AN^{\dagger}_{\g} X,E)
	\]
which is an isomorphism.

\end{theorem}

\section{Excision on topoi}
\label{S:topoiReview}

In this section we exposit a small piece of \cite{SGA4:I}.

\subsection{Topoi}

Here we recall definitions and basic facts about categories, presheaves, sheaves, sites, topoi, localization, fibered categories, and 2-categories. We refer to \cite{stacks-project} (and its prodigious index and table of contents) for any omitted details and a more leisurely and complete discussion of these concepts, and in particular follow their convention that a left exact functor is defined to be a functor that commutes with finite limits and a right exact functor is a functor that commutes with colimits (see \cite{stacks-project}*{\href{http://math.columbia.edu/algebraic_geometry/stacks-git/locate.php?tag=0034}{0034}}).
\\

\labelpar{}
\label{p:initial}
Let $C$ be a category. We denote by $\widehat{C}$ the category $\Fun(C^{\op},\Sets)$  of \defi{presheaves} on $C$. We denote by $h\colon C \to \widehat{C}$ the Yoneda embedding which sends an object $X$ of $C$ to the presheaf $h_X := \Hom(-,X)$. We say that a presheaf $F \in \widehat{C}$ is \defi{representable} if there exists an $X \in C$ and an isomorphism $h_X \to F$, and we say that $F$ is \defi{representable by} $X$ if $F$ is isomorphic to $h_X$. The functor $h$ is fully faithful, and so when there is no confusion we will consider $C$ as a full subcategory of $\widehat{C}$; i.e., we will identify $h_X$ with the object $X$ that it represents.

Similarly, we say an object $X \in C$ \defi{corepresents} a covariant functor $F\colon C \to \Sets$ if $F$ is isomorphic to the functor $Y \mapsto \Hom(X,Y)$.
\\

\labelpar{}
\label{p:localization} 
Let $X \in C$ be an object. We define the \defi{localized} (or `comma') category $C_{/X}$ to be the category of maps $Y \to X$ whose morphisms are commuting diagrams 
\[
\xymatrix{
Y\ar[rr]\ar[dr] && Y'\ar[dl]\\
&X&
}.
\]
There is a projection functor $j_X\colon C_{/X} \to C$ which we denote by $j$ when the context is clear. 

\labelpar{} Let $\xymatrix{ C  \ar@<+.5ex>[r]^<>(.5){L}  & D \ar@<+.5ex>[l]^<>(.5){R} }$ be a pair of functors between categories $C$ and $D$. We say that $L$ is \defi{left adjoint} to $R$ (or equivalently that $R$ is \defi{right adjoint} to $L$) if there is a natural isomorphism 
\[
\Hom(L(-),-) \cong \Hom(-,R(-))
\]
of bifunctors. The natural transformation $\id\colon L \to L$ (resp. $\id\colon R \to R$) induces (via the adjunction) a functor $\id_C \to R\circ L$ (resp. $L \circ R \to \id_D$) called the \defi{unit} (resp. \defi{counit}) of adjunction.

\begin{lemma}
\label{L:unitFull}

The functor $L$ (resp. $R$) is fully faithful if and only if the unit (resp. counit) of adjunction is an isomorhpism.

\end{lemma}

\begin{proof}

Let $Y \in C$. By adjunction, for any $X \in C$, the second morphism of the composition
\[
\Hom(X,Y) \to \Hom(X,R(L(Y))) \to \Hom(L(X),L(Y)).
\]
is an isomorphism. By definition the composition is an isomorphism for all $Y$ if and only if $L$ is fully faithful, and by Yoneda's lemma, the first map is an isomorphism for all $Y$ if and only if the unit of adjunction is an isomorphism. The second claim is proved in the same way using the co-Yoneda lemma.

\end{proof}

\labelpar{}
\label{p:adjointTriple}
Let $C$ and $D$ be categories, and let $u\colon C \to D$ be a functor. Then from $u$ we can construct a triple $\widehat{u}_!,\widehat{u}^*,\widehat{u}_*$ of functors 
 \[
  \xymatrix@C=2cm@M=.25cm{
	\widehat{C}
        \ar@<7pt>[r]^{\widehat{u}_!}\ar@<-7pt>[r]_{\widehat{u}_*} &
	\widehat{D}
		\ar[l]|{\widehat{u}^{*}} 
  }
  \]
with each left adjoint to the functor directly below. The functor $\widehat{u}^*$ is the easiest to define, and sends a presheaf $G \in \widehat{D}$ to the presheaf $\widehat{u}^*G := G \circ u$ on $C$ (i.e., the presheaf $X \mapsto F(u(X))$. To construct a left adjoint $\widehat{u}_!$ one first observes that for $X \in C$ one is forced by the adjunction 
\[
\Hom(\widehat{u}_!\, h_X,F) = \Hom(h_X,\widehat{u}^*F) = (\widehat{u}^* F)(X) = F(u(X))
\]
to define $\widehat{u}_!(h_X) = h_{u(X)}$. Every sheaf $F \in C$ is isomorphic to a colimit of representable sheaves via the natural map $\colim_{h_X \to F} h_X \to F$, where the colimit is taken over the comma category $C_{/F}$ whose objects are maps $h_X \to F$ and whose morphisms are commuting diagrams of maps. One's hand is again forced -- since a functor with a right adjoint is right exact, $\widehat{u}_!$ should commute with colimits and we are forced to define $\widehat{u}_!F$ as $\colim_{h_X \to F} h_{u(X)}$. Alternatively, a rearrangement gives the usual formula (see for instance \cite{stacks-project}*{\href{http://math.columbia.edu/algebraic_geometry/stacks-git/locate.php?tag=00VD}{00VD}})
\[
Y \mapsto \colim_{h_X \to F} (h_{u(X)}(Y)) \cong \colim_{h_X \to F} \colim_{|h_{u(X)}(Y)|} * \cong
\]
\[
 \colim_{X \in (I^Y_u)^{\op}}\colim_{|h_X \to F|}*  \cong \colim_{X \in (I^Y_u)^{\op}} F(X),
\]
where $* = \{\emptyset\}$, $|C|$ denotes the underlying set of a category $C$, $I^Y_u$ is the category whose objects are pairs $(X, Y \to u(X))$ and whose morphisms are morphisms $X \to X'$ which make the diagram
\[
\xymatrix{
&Y\ar[rd]\ar[ld]&\\
u(X)\ar[rr]&&u(X')
}
\]
commute, and the colimit is taken in the category of sets. Later it will be important to observe that when $F(X)$ has extra algebraic structure (e.g., $F$ is a sheaf of abelian groups), we can take this colimit in a different category and construct a different left adjoint $\widehat{u}_!$. If the category $(I^Y_u)^{\op}$ is directed then $\widehat{u}_!$ is exact, but this does not hold in general. By construction it is left adjoint to $\widehat{u}^*$.

The functor $\widehat{u}_*$ is easier to construct -- by adjunction we can define for $Y \in D$ and $F \in \widehat{C}$  value of the presheaf $\widehat{u}_*F$ on $X$ as
\[
(\widehat{u}_*F)(Y) = \Hom(h_Y,\widehat{u}_* F) = \Hom(\widehat{u}^* \, h_Y, F);
\]
and writing $\widehat{u}^*\, h_Y$ as a colimit of representable presheaves we deduce a description of $\widehat{u}_*F$ as the presheaf $Y \mapsto \lim_{u(X) \mapsto Y} F(X)$. Any functor with a left (resp. right) adjoint commutes with arbitrary limits (resp. colimits) when the limits exist \cite{stacks-project}*{\href{http://math.columbia.edu/algebraic_geometry/stacks-git/locate.php?tag=0038}{0038}}. Thus, $\widehat{u}_*$  commutes with limits, and  $\widehat{u}^*$ commutes with both limits and colimits.
\\

\begin{example} 
\label{E:topologicalExtensionByZero}

Let $X$ be a topological space, let $\Open X$ be the category of open subsets of $X$, and consider the inclusion $i\colon \Open U \hookrightarrow \Open X$ induced by the open inclusion of topological spaces $U \subset X$. Then the morphisms $\widehat{i}^*$ and $\widehat{i}_*$ are the usual morphisms (induced by the alternative functor $\Open X \to \Open U$ given by intersection), and $\widehat{i}_!$ is the `extension by the empty set' functor, so that $\widehat{i}_!F$ is given by 
\[
U' \mapsto \begin{cases} F(U') & \text{if $U' \subset U $,}
\\
\emptyset &\text{if $U' \not \subset U$.}
\end{cases}
\]
Finally, we note that for any category $C$, the category $\widehat{C}$ has a final object $\e_{\widehat{C}}$ given by the presheaf $X \mapsto \{\emptyset\}$; this is also a limit of the empty diagram. Since left exact functor $\widehat{C} \to E$ with $E$ a category must send $\e_{\widehat{C}}$ to a final object of $E$, we conclude that the functor $\widehat{i}_!$ is not left exact.

\end{example}

\labelpar{}
\label{p:presheafLimits}
Let $I$ and $D$ be categories. For $Y \in D$, define $F_Y\colon I \to D$ to be the constant functor $i \mapsto Y$. Let $F\colon I \to D$ be a functor. We say that $X$ is a \defi{limit} of $F$ if $X$ represents the functor $Y \mapsto \Hom(F_Y,F)$, and we say that $X$ is a \defi{colimit} of $F$ if it corepresents the functor $Y \mapsto \Hom(F,F_Y)$. We will often refer to $F$ as a diagram.

When $D$ is the category of sets, limits and colimits exist. It follows that when $D$ is the category $\widehat{C}$ of presheaves on a category $C$, limits and colimits exist -- indeed, the limit (resp. colimit) of a diagram $F\colon I \to \widehat{C}$ of presheaves is the presheaf sending $X$ in $C$ to the limit (resp. colimit) of the diagram $\text{ev}_X \circ F$ (i.e., the functor given by $i \mapsto I(i)(X)$).
\\

\labelpar{}
\label{p:presheafUnions}
In particular, let $I$ be a category whose only morphisms are the identity morphisms, and let $\{X_i\}_{i \in I}$ be a collection of objects of $\widehat{C}$. Then the colimit of the diagram $i \mapsto X_i$, which we call the \defi{disjoint union} of $\{X_i\}$ and denote by $\coprod_{i \in I} X_i$, exists in $\widehat{C}$. Moreover, coproducts commute with localization; i.e., if we define $\coprod C_{/X_i}$ to be the 2-categorical (see \ref{R:2-fiberProducts}) fiber product $I\times_C \Mor C$ via the map $F\colon I \to C$, then the natural map 
\[
 \coprod C_{/X_i}  \to C_{/\coprod X_i},
\]
is an equivalence of categories.
\\

\labelpar{}
Let $X \in C$ be an object of a category $C$ and consider the projection morphism $j_X\colon C_{/X} \to C$ (see \ref{p:localization}). One can make the triple of adjoint functors of \ref{p:adjointTriple} more explicit as follows. The collection of maps $Y \to X$ is cofinal in $(I^Y_j)^{\op}$, and so the functor $\widehat{j}_!$ may be concisely described as sending a presheaf $F \in \widehat{C_{/X}}$ to the presheaf 
\begin{equation}
\label{eq:extensionFormula} 
 \widehat{j}_!F\colon Y \mapsto \coprod_{Y \to X}F(Y \to X).
\end{equation}
Alternatively, the presheaf category $\widehat{C_{/X}}$ is canonically isomorphic to the localization $\widehat{C}_{/h_X}$ via the map $\widehat{C}_{/h_X} \to \widehat{C_{/X}}$ which sends $F \to h_X$ to the presheaf $(Y \to X) \mapsto \Hom_{h_X}(h_Y,F)$; the inverse map is $F \mapsto (\widehat{j}_!F \to h_X)$ ($\widehat{C_{/X}}$ has a final object represented by $\id\colon X \to X$, and the map to $h_X$ is $\widehat{j}_!$ of the map from $F$ to the final object). Via this identification the functor $\widehat{j}_!$ simply sends a presheaf $F \to h_X$ to $F$, and the map $u^*F$ sends a presheaf $F \in \widehat{C}$ to the product $h_X \times F \to h_X$ (where the map is the first projection).

For $F \in \widehat{C}$, we define the localization $C_{/F}$ similarly, by the formula $C_{/F} := C \times_{\widehat{C}}\left(\widehat{C}_{/F}\right)$.
\\

\labelpar{}
Let $u\colon C \to D$ be a functor. We say that an arrow $Y \to X$ of $C$ is \defi{cartesian} if for any $\psi:Z\to X$ and for any $h\colon u(Z)\to u(Y)$
   such that $u(\psi)=u(\phi)\circ h$, there exists a unique $\theta\colon Z\to Y$ so
   that $\psi=\phi\circ \theta$.
   \[\xymatrix@R-1.5pc{
    Z \ar@{|->}[dd]\ar@{-->}[dr]_{\exists! \theta} \ar@/^/[rrd]^{\forall \psi}\\
     & Y\ar@{|->}[dd] \ar[r]_\phi & X \ar@{|->}[dd]\\
    u(Z) \ar[dr]_{\forall h} \ar@/^/[rrd]^(.65){u(\psi)}|!{[ru];[rd]}\hole \\
    & u(Y) \ar[r] & u(X)
   }\]
and we say that $u$ (or when the base $D$ is clear, `$C$') is a \defi{fibered category} or a \defi{category fibered over} $D$ if for every $X \in C$ and every arrow $Y \to u(X)$ in $D$, there exists a cartesian arrow over $Y \to u(X)$.

For $X \in D$ we define the \defi{fiber over} $X$ to be the category $C(X) := u^{-1}(\id\colon X \to X)$ of all objects of $C$ which map to $X$ with morphisms which map to the identity $\id\colon X \to X$ under $u$. If for every $X$, the category $C(X)$ is a groupoid (i.e., a category such that every arrow is an isomorphism), then we call $C$ a \defi{category fibered in groupoids over} $D$. In this case every arrow of $C$ is cartesian.
\\

\begin{example} 
\label{E:fiberedExamples}

Let $C$ be a category and let $F \in \widehat{C}$ be a presheaf. Then the comma category $j\colon C_{/F} \to C$ is a category fibered in groupoids; in fact it is fibered in sets (i.e., categories such that every arrow is the identity), and any category fibered in sets over a category $C$ is equivalent (but not necessarily isomorphic) to a fibered category $C_{/F}$ for some $F \in \widehat{C}$.

Let $C$ be a category with fiber products. Another example of a fibered category is the codomain fibration $\Mor C \to C$: objects of $\Mor C$ are morphisms of $C$ and arrows are commutative diagrams, and the map $t\colon\Mor C \to C$  sends an arrow $Y \to X$ to its target $X$. Then for $X \in C$, the comma category is equal to $(\Mor C)(X)$.

\end{example}

\labelpar{}
Categories fibered over $C$ form a 2-category, i.e., a category enriched over categories (so that $\Hom(X,Y)$ is not just a set, but a category). An element of $\Mor (\Hom(X,Y))$ is called a 2-morphism. Let $X,Y$ be two categories fibered over $C$. A \defi{morphism} of categories fibered over C is a functor $F\colon X \to Y$ such that the diagram 
\[
\xymatrix{
X \ar[rr]\ar[dr] &&Y\ar[dl]\\
&C&
}
\]
commutes and $F$ takes cartesian arrows to cartesian arrows (if $X$ and $Y$ are fibered in groupoids, then every arrow is cartesian, so this last condition is automatic). A 2-morphism between morphisms $F,R\colon X\to Y$ is a natural transformation $t\colon F \to R$ such that for every $x \in X$, the induced map $t_x\colon F(x) \to R(x)$ in $Y$ projects to the identity morphism in $C$. One can check that when $X$ and $Y$ are fibered in groupoids, any 2-morphism is actually an isomorphism.

\begin{remark}
\label{R:2-fiberProducts}
  A main point of the use of the formalism of 2-categories is that equivalence of categories is not respected by fiber products of categories. Instead one considers 2-categorical fiber products, defined as in \cite{stacks-project}*{\href{http://math.columbia.edu/algebraic_geometry/stacks-git/locate.php?tag=02X9}{02X9}}).
\end{remark}

\labelpar{}
\label{p:topology}
Let $C$ be a category. We define a \defi{pretopology} (often called a Grothendieck Topology) on $C$ to be a set $\Cov C$ of families of morphisms (which we call the coverings of $C$) such that each element of $\Cov C$ is a collection $\{X_i \to X\}_{i \in I}$ of morphisms of $C$ with a fixed target satisfying the usual axioms (see \cite{stacks-project}*{\href{http://math.columbia.edu/algebraic_geometry/stacks-git/locate.php?tag=00VH}{00VH}}): 
\begin{itemize}
\item[(i)] For every isomorphism $X \cong X'$, $\{X \cong X'\} \in \Cov C$;
 
\item[(ii)] Refinements of a covering by coverings form a covering;

\item[(iii)] For every $\{X_i \to X\}_{i \in I} \in \Cov C$ and every $Y \to X$, each of the fiber products $X_i \times_X Y$ exists and $\{X_i \times_X Y \to Y\}_{i \in I} \in \Cov C$.
\end{itemize}
We call a category $C$ with a pretopology $\Cov C$ a \defi{site}. This generates a \defi{topology} on $C$ in the sense of \cite{stacks-project}*{\href{http://math.columbia.edu/algebraic_geometry/stacks-git/locate.php?tag=00Z4}{Definition 00Z4}}.
\\

\labelpar{}
\label{p:sheaf}
Let $C$ be a site whose topology is defined by a pretopology and let $F \in \widehat{C}$ be a presheaf. We say $F$ is a \defi{sheaf} if for every covering $\{X_i \to X\}_{i \in I} \in \Cov C$ the diagram
\[
\xymatrix{
F(X) \ar[r]
&
\prod\nolimits_{i\in I}
F(X_i)
\ar@<1ex>[r]^-{\text{pr}_0^*} \ar@<-1ex>[r]_-{\text{pr}_1^*}
&
\prod\nolimits_{(i_0, i_1) \in I\times I}
F(X_{i_0}\times_X X_{i_1})
}
\]
is exact (i.e., the first arrow equalizes the rest of the diagram). We denote by $\widetilde{C}$ the category of sheaves on $C$.

The inclusion functor $i\colon \widetilde{C} \hookrightarrow \widehat{C}$ has a left adjoint, `sheafification', which we denote by $-^a$. In particular, the inclusion $i$ commutes with limits (but not colimits!), so that the limit of a diagram of sheaves in the category of sheaves agrees with the limit considered in the category of presheaves (i.e., limits do not require sheafification in $\widetilde{C}$). We conclude that $\widetilde{C}$ has a final object $\e_{\widetilde{C}}$, which is the limit of the empty diagram and given (as in the case of presheaves) by the sheaf $X \mapsto \{\emptyset\}$. 
\\

\labelpar{}
\label{p:sheafLimits}
Limits exist in $\widetilde{C}$ -- indeed, given a diagram $F\colon I \to \widetilde{C}$, the limit of the diagram $i \circ F\colon I \to \widehat{C}$ is a sheaf and thus the limit of the diagram $F$. Colimits in $\widetilde{C}$ also exist -- the colimit of a diagram $F$ is the \emph{sheafication} of the diagram $i \circ F$ (an example where sheafification is required is a disjoint union of topological spaces). 
\\

\labelpar{}
A \defi{topos} is a category equivalent to the category $\widetilde{C}$ of sheaves on a site $C$. A morphism $f\colon T' \to T$ of topoi is a pair $(f^*\colon T \to T', f_*\colon T' \to T)$ of functors such that $f^*$ is exact and left adjoint to $f_*$.
\\

\labelpar{}
Let $C$ and $D$ be sites and let $u \colon C \to D$ be a functor. Then the functors $\widehat{u}_!, \widehat{u}^*,$ and $\widehat{u}_*$ do not necessarily restrict to maps between $\widetilde{C}$ and $\widetilde{D}$ (i.e., they do not necessarily send sheaves to presheaves), and if we sheafify then they may no longer be adjoint. This motivates the following  definitions. 

We say that $u$ is \defi{continuous} if $\widehat{u}^*$ of a sheaf is a sheaf, and in this case we denote the induced map $\widetilde{D} \to \widetilde{C}$ by $u^*$. If the topology on $C$ is defined by a pretopology and $u$ commutes with fiber products, then by \cite{stacks-project}*{\href{http://math.columbia.edu/algebraic_geometry/stacks-git/locate.php?tag=00WW}{00WW}}, $u$ is continuous if and only if it sends coverings of $C$ to coverings of $D$. Note that we generally do not expect that $u$ commutes with \emph{arbitrary} finite limits -- consider for example an object $X \in C$ and the projection morphism $C_{/X} \to C$. If in addition $\widehat{u}_!$ is exact, we then say that $u$ is a \defi{morphism of sites}; setting $u_! = (\widehat{u}_!)^a$ it follows that the pair $(u_!, u^*)\colon \widetilde{C} \to \widetilde{D}$ is a morphism of topoi.

Alternatively, we say that a functor $u\colon C \to D$ is \defi{cocontinuous} if $\widehat{u}_*$ sends sheaves to sheaves, and in this case we denote the induced map $\widetilde{C} \to \widetilde{D}$ by $u_*$. The pair $(u^*, u_*) \colon \widetilde{C} \to \widetilde{D}$ is then a morphism of topoi, where $u^*$ is the sheafification $(\widehat{u}^*)^a$. If the topology on $D$ is defined by a pretopology, then by \cite{stacks-project}*{\href{http://math.columbia.edu/algebraic_geometry/stacks-git/locate.php?tag=00XK}{00XK}} $u$ is cocontinuous if and only if for every $X \in C$ and every covering $\{Y_j \to u(X)\}_{j \in J}$ of $u(X)$ in $D$ there exists a covering $\{X_i \to X\}_{i\in I}$ in $C$ such that the family of maps $\{u(X_i) \to u(X)\}_{i \in I}$ refines the covering $\{Y_j \to u(X)\}_{j \in J}$, in that the collection $\{u(X_i) \to u(X)\}_{i \in I}$ is a covering of $u(X)$ and that there is a map $\phi\colon I \to J$ such that for each $i$ there exists a factorization $u(X_i) \to Y_{\phi(i)} \to u(X)$ (note that we do not require the collections $\{u(X_i) \to Y_{j}\}_{\phi(i) = j}$ to be coverings).

The nicest situation is when $u \colon C \to D$ is both continuous and cocontinuous -- the induced morphism $(u^*,u_*)\colon \widetilde{C} \to \widetilde{D}$ requires no sheafication and $u^*$ has a left adjoint.
\\

\labelpar{}
\label{p:inducedTopology}
Let $u\colon C \to D$ be a functor, and suppose that $D$ is a site. We define the \defi{induced topology} on $C$ to be the largest topology making the map $u$ continuous. When $u$ commutes with fiber products and the topology on $D$ is defined by a pretopology, then the induced topology on $C$ is generated by the following pretopology: a collection $\{V_i \to V\}$ in $C$ is a covering if $\{u(V_i) \to u(V)\}$ is a covering in $D$. 

Now suppose instead that $C$ is a site. We define the \defi{image topology} on $D$ to be the smallest topology making the map $u$ continuous. When $u$ commutes with fiber products, the topology on $C$ is defined by a pretopology, then the image topology on $D$ is generated by the following pretopology: for every covering $\{V_i \to V\}$ in $C$, the collectin $\{u(V_i) \to u(V)\}$ is a covering in $D$.

\begin{example}
\label{E:cocontinuous}

 Our main example of a cocontinuous functor is the following. Let $D$ be a site and let $u\colon C \to D$ be a fibered category such that every arrow of $C$ is cartesian, and endow $C$ with the induced topology \ref{p:inducedTopology}. Assume further that finite limits exist in $C$ and $D$ and that the topology on $D$ is defined by a pretopology. Since $u$ is fibered in groupoids, it is an easy exercise to check  that $u$ commutes with fiber products. Then it follows immediately from the definitions (using that $u$ is a fibered category) that $u$ is cocontinuous, and we get a triple of adjoints
 \[
  \xymatrix@C=2cm@M=.25cm{
	\widetilde{C}
		\ar@<7pt>[r]^{{u}_{!}}\ar@<-7pt>[r]_{{u}_*} &
	\widetilde{D}
		\ar[l]|{{u}^{*}}.
  }
  \]
We will mainly apply this when $C \cong D_{/X}$ for some $X \in D$.

\end{example}

\labelpar{p:canonicalTopology}
Let $C$ be a category. We define the \defi{canonical topology} on $C$ to be the largest topology such that representable objects are sheaves (i.e., the largest topology  such that for all $x \in C$, the presheaf $h_X$ is a sheaf). We say that a topology is \defi{subcanonical} if it is smaller than the canonical topology (in other words, for all $x \in C$, the presheaf $h_X$ is a sheaf). 

\begin{example}
\label{E:canonicalExamples}

(i) The topology on the category of affine schemes given by jointly surjective families of flat (but not necessarily finitely presented) morphisms is subcanonical \cite{Knutson:algebraicSpaces}*{3.1, 7'}, and the fpqc topology is subcanonical on the category of schemes \cite{Vistoli:fiberedCategories}*{Theorem 2.55} (note that the flat topology is \emph{not} subcanonical for the category of schemes).

(ii) For a site $C$ the canonical topology on $\widetilde{C}$ is given by collections $\{F_i \to F\}$ such that the map $\coprod F_i \to F$ is a surjection of sheaves \cite{stacks-project}*{\href{http://math.columbia.edu/algebraic_geometry/stacks-git/locate.php?tag=03A1}{03A1}}. The natural map $\widetilde{C} \to \widetilde{\widetilde{C}}$ is then an equivalence of categories. Thus, any topos $T$ is canonically a site.

\end{example}

\labelpar{}
For a topos $T$ we denote by $\Ab T$ the category of abelian group objects of $T$. 
If we view $T$ as a site with its subcanonical topology, then  $\Ab T$ is equivalent to the category of sheaves of abelian groups on $T$, and when we choose a site $C$ such that $T$ is equivalent to $\widetilde{C}$, we may write $\Ab C$ instead of $\Ab T$.  By \cite{stacks-project}*{\href{http://math.columbia.edu/algebraic_geometry/stacks-git/locate.php?tag=00YT}{00YT}}, a morphism of $f\colon T' \to T$ topoi restricts to a pair
 \[
  \xymatrix@C=1cm@M=.115cm{
	\Ab T'
		\ar@<-4pt>[r]_{f_*} &
	\Ab T
		\ar@<-4pt>[l]_{f^*}
  }
  \]
of adjoint functors; here the exactness of $f^*$  in the definition of a morphism of topoi is crucial (consider for example that the functor $u_!\colon \widetilde{C_{/X}} \to \widetilde{C}$ described above in \ref{E:topologicalExtensionByZero} is not generally exact and indeed fails to send abelian sheaves to abelian sheaves).
\\

\labelpar{}
\label{p:abelianExtensionZero}
Let $u\colon C \to D$ be a morphism of sites. Then $u_!$ does not necessarily take abelian sheaves to abelian sheaves.  Indeed, consider the case of a localization morphism $j\colon C_{/X} \to C$ (with $X \in C$). Then for any $X' \in C$ such that $\Hom(X',X)$ is empty, $(j_!F)(X')$ is also empty for any abelian sheaf $F \in \Ab \widetilde{C_{/X}}$. It is nonetheless true that $u^*\colon \Ab \widetilde{D} \to \Ab \widetilde{C}$  has a left adjoint $u_!^{\ab}$. We will construct $u_!^{\ab}$ in the next few paragraphs by adapting the construction of $u_!$.

As a first step we consider a category $C$ and construct a left adjoint $\mathbb{Z}^{\text{ps}}_{-}\colon C \to \Ab C$ to the forgetful functor $\Ab C \to C$. Let $F \in \widehat{C}$ be a presheaf of sets. We define the \defi{free abelian presheaf} on $F$ to be the presheaf $X \mapsto \bigoplus_{s \in F(X)} \mathbb{Z}$. It follows directly from this explicit formula that this is the desired left adjoint and, moreover, that the functor $F \mapsto \mathbb{Z}^{\text{ps}}_F$ commutes with limits; since it has a right adjoint it also commutes with colimits and is thus exact. When $F = h_X$ for some $X \in C$, we will instead write $\mathbb{Z}^{\text{ps}}_X$.

Now, suppose that $C$ is a site. Since sheafification is left adjoint to the inclusion $\widetilde{C} \hookrightarrow \widehat{C}$, the functor $\mathbb{Z}_{-}\colon \widetilde{C} \to \Ab \widetilde{C}$ given by $F \mapsto (\mathbb{Z}_F)^{a}$ is left adjoint to $\Ab \widetilde{C} \to \widetilde{C}$. Furthermore, since sheafification is exact, the functor $\mathbb{Z}_{-}$ also commutes with limits and colimits. When $F = (h_X)^{a}$ for some $X \in C$, we will instead write $\mathbb{Z}_X$.

Now we can construct $u_!^{\ab}$ as following the template of \ref{p:adjointTriple}. Let $A \in \Ab C$ be a sheaf of abelian groups and let $U \in C$. Then since
\[
A(U) = \Hom_{\widetilde{C}}(h_U, A) = \Hom_{\Ab \widetilde{C}}(\mathbb{Z}_U,A),
\]
and since $u_!^{\ab}$ commutes with colimits, it follows that 
\[
u_!^{\ab}A = 
u_!^{\ab}\colim_{(h_U \to A) \in \widetilde{C}_{/A}} h_{h_U \to A} =   
\colim_{(h_U \to A) \in \widetilde{C}_{/A}}  \mathbb{Z}_{h_U \to A}.
\]
As in the case of $u_!$ for sheaves of sets, by adjunction we must have $u_!^{\ab} \mathbb{Z}_{h_U \to A} = \mathbb{Z}_U$, and since $u_!^{\ab}$ must commute with colimits we get the formula 
\[
u_!^{\ab} A =  \colim_{(h_u \to A) \in \widetilde{C}_{/A}} \mathbb{Z}_{U}.
\]
As before (see Equation \ref{eq:extensionFormula}) we get a nice formula when $u = j_X\colon C_{/X} \to C$ is the projection morphism associated to some object $X$ of a site $C$ (see \ref{p:localization}); $u_!^{\ab}A$ is the sheafication of the presheaf
\begin{equation}
\label{eq:j_!}
 \widehat{j}_!A\colon Y \mapsto \bigoplus_{Y \to X}A(Y \to X).
\end{equation}
In this special case it follows from this explicit formula that $u_!^{\ab}$ left exact; moreover it commutes with colimits since it has a right adjoint $u^*$. Consequently, by an easy exercise we get the useful bonus that $u^*$ takes injective abelian sheaves to injective abelian sheaves.

Note that this disagrees with the functor `extension by the empty set' $u_!$; nonetheless when there is no confusion we will write $u_!^{\ab}$ as $u_!$ (and if there is confusion we will refer to them by $u_!^{\ab}$ and $u_!^{\set}$).
\\

\labelpar{}
A \defi{ringed topos} is a pair $(T,\calO_T)$ with $T$ a topos and $\calO_T$ a ring object of $T$. Equivalently, $\calO_T$ is a sheaf of rings on $T$, where we consider $T$ as a site with its canonical topology  (see Definition \ref{p:canonicalTopology}). A morphism $f\colon (T', \calO_{T'}) \to (T,\calO_T)$ of ringed topoi is a morphism $f\colon T' \to T$ of topoi and a map $\calO_T \to f_*\calO_{T'}$. Sometimes we will write $(f^*,f_*)$ instead of $f$. Similarly, a \defi{ringed site} is a site $C$ together with a ring object $\calO_C$ of its topos $\widetilde{C}$, and a morphism $(C', \calO_{C'}) \to (C,\calO_C)$ of ringed sites is a continuous morphism $f\colon C' \to C$ of sites and a map $\calO_C \to f_*\calO_{C'}$.
\\

\labelpar{}
Let $(T,\calO_T)$ be a ringed topos. Then we can consider the category $\Mod \calO_T$ of $\calO_T$-modules (i.e., the category of abelian group objects of $T$ which admit the structure of a module object over the ring object $\calO_T$ of $T$). Considering $T$ with its canonical topology, an $\calO_T$-module is the same as a sheaf of $\calO_T$-modules.

We say that $M \in \Mod \calO_T$ is \defi{quasi-coherent} (resp. \defi{locally finitely presented}) if there exists a covering $F \to \e_T$ of the final object $e_T$ 
of $T$ (i.e., a covering in the canonical topology on $T$) such that, denoting by $j\colon T_{/F} \to T$ the localization with respect to $F$ and setting $\calO_{T_{/F}} = j^*\calO_T$, the pullback $j^*M$ admits a presentation (resp. finite presentation) -- i.e., $j^*M$ is the cokernel of a map $\bigoplus_I \calO_{T_{/F}} \to \bigoplus_J \calO_{T_{/F}}$ of $\calO_{T_{/F}}$-modules (resp. a map with $I$ and $J$ finite sets). If $C$ is a site such that $T$ is equivalent to $\widetilde{C}$ (which may have no final object) and the topology on $C$ is defined by a pretopology, then it is equivalent to ask that for all $X \in C$, there exists a covering $\{X_i \to X\}$ such that for each $i$ there exists a presentation (resp. finite presentation) of the restriction of $M$ to $T_{(h_{X_i})^a}$. We denote by $\QCoh \calO_T$ (resp. $\Mod_{\fp} \calO_T$) the subcategories of quasi-coherent (resp. locally finitely presented) $\calO_T$-modules.
\\

\labelpar{}
\label{p:modFunctors}
Let $(T,\calO)$ be a ringed topos, with $T = \widetilde{C}$ for some site $C$. As in the abelian case, the forgetful functor $\Mod \calO \to T$ has a left adjoint $\calO_{-}\colon T \to \Mod \calO$. When $T$ is equivalent to $\widetilde{C}$ for some site $C$ and $\calO$ is a sheaf on $C$, then for $F \in T$, $\calO_F$ is defined in the same manner as $\mathbb{Z}_F$: $\calO_F$ is  the sheafification of the presheaf $X \mapsto \bigoplus_{s \in F(X)}\calO(X)$. As usual, when $F = (h_U)^a$ for  $U \in C$ we denote $\calO_{F}$ by $\calO_U$.

Let $u\colon (C, \calO_C) \to (D, \calO_D)$ be a morphism of ringed sites. Then the above template for the construction of $j_!^{\ab}$ admits a verbatim translation (replacing the free functor $\mathbb{Z}_{-}$ by the free functor $\calO_{-}$) and allows one to construct a left adjoint $u_!^{\Mod}\colon \Mod \calO_C \to \Mod \calO_D$ to the functor $u^*(-)\otimes_{u^*\calO_D}\calO_C \colon \Mod \calO_D \to \Mod \calO_C$. When $u = j_X\colon C_{/X} \to C$ is the projection morphism associated to some object $X$ of a site $C$ and $C_{/X}$ has the induced topology, $u_!^{\Mod}$ is even defined by the same formula \ref{eq:j_!} as $u_!^{\ab}$. Again we will denote $u_!^{\Mod}$ by $u_!$.

\subsection{Excision}
\label{S:excisionTopoi}

Here we recall very general facts about immersions of topoi, open and closed sub-topoi, excision, and cohomology supported in a closed sub-topos.
\\

Let $f \colon T' \to T$ be a morphism of topoi. We say that $f$ is an \defi{immersion} if $f_*$ is fully faithful \cite{SGA4:I}*{Definition 9.1.2}; by Yoneda's lemma this is equivalent to the adjunction $\id \to f^{-1}f_*$ being an isomorphism.

Let $T$ be a topos. Then $T$ has a final object (see \ref{p:sheaf}), a choice of which we denote by $\e_T$. Following \cite{SGA4:I}*{Definition 8.3}, we say that an object $U \in T$ is \defi{open} if it is a subobject of $\e_T$ (i.e., if the map $U \to \e_T$ is a monomorphism). Similarly, for $X \in T$ we define an \defi{open} of $U \subset X$ to be an open object $U \subset T_{/X}$ of the topos $T_{/X}$.  

Let $U \in T$ be open. The restriction map $j\colon T_{/U} \to T$ induces a morphism $(j^*,j_*)$ of topoi, which is an immersion -- indeed, using the explicit description of the pair $(j^*,j_*)$ and that $U \to \e_T$ is a monomorphism one can easily check that the adjunction is an isomorphism. We define $T' \to T$ to be an \defi{open immersion} of topoi if it is isomorphic to $T_{/U} \to T$ with $U \in T$ open, and we say that $T' \to T$ is an open subtopos, and we say that a morphism of sites is an open immersion if the induced morphism of topoi is an open immersion.

Now let $T' \to T$ be an open immersion, and let $U \in T$ be an open such that $T' \to T$ is isomorphic to $T_{/U} \to T$. As in \cite{SGA4:I}*{9.3.5}, we define the \defi{closed complement} $Z$ of $T'$ in $T$ to be the \emph{complement} of $T_{/U}$ in $T$, i.e., the largest sub-topos $Z$ of $T$ such that $T_{/U} \cap Z$ is equivalent to $\{\e_{T_{/U}}\}$. Concretely, $Z$ is the full subcategory of objects $F \in T$ such that the projection map $U \times F \to U$ is an isomorphism (i.e., such that $j^*F$ is isomorphic to $\e_{T_{/U}}$). The category $Z$ is independent of the choice of $U$. When $T' = T_{/U}$, we will also call $Z$ the closed complement of $U$.

We denote the inclusion $Z \hookrightarrow T$ by $i_*$ and remark that by \cite{SGA4:I}*{Proposition 9.3.4}, the map $i^*\colon T \to Z$ given by $F \mapsto U \coprod_{U \times F} F$ is adjoint to $i_*$, with adjunction given by the projection morphism $F \to U \coprod_{U   \times F} F$, and that together these form a morphism $(i^*,i_*)\colon Z \to T$ of topoi. Since $Z$ is a full subcategory, the inclusion $i$ is an immersion of topoi, and we say that any immersion of topoi isomorphic to an immersion $Z \to T$ arising as the closed complement of an open immersion is a \defi{closed immersion} of topoi and say that $Z$ is a closed sub-topos of $T$.

\begin{remark}
\label{R:geometricImmersions}

  Let $C$ be a site. Let $X \to X'$ be a monomorphism in $C$. Then the induced map $\widetilde{C}_{/X} \to \widetilde{C}_{/X'}$ is an open immersion. In particular, when $C = \Sch$, two odd examples of `open immersions' arise from $X \to X'$ a closed immersion or $\Spec \calO_{X',\, x} \to X'$, with $x \in X'$! Remarkably, as above one can still define a notion of `closed complement' of a closed immersion and deduce an excision theorem (see Proposition \ref{P:excisionRelations}). 

Of course, the more interesting open immersions are those whose closed complements admit a `geometric' description. For example, if $U \subset X$ is an open inclusion of topological spaces, then $U$ is an object of the site $\Open X$ and the restriction morphism $j\colon \Open U \cong (\Open X)_{/U} \to \Open X$ induces the usual morphism of topoi induced by the continuous morphism of sites $\Open X \to \Open U,\, U' \mapsto U \cap U'$. If we denote by $Z$ the (topological) complement of $U$ in $T$, then the closed complement of $U$ in $X_{\Open}$ is isomorphic to the usual inclusion induced by the continuous morphism of sites $\Open X \to \Open Z$ given again by intersection.

Another `geometric' example is le Stum's explication of Berthelot's $j^{\dagger}$ functor \cite{leStum:rigidBook}*{Proposition 5.1.12 (a)}; see Remark \ref{R:jDaggerDescription}.

\end{remark}

A closed immersion $Z \to T$ of topoi enjoys many of the same properties as the classical case $\Open Z \to \Open X$; see \cite{SGA4:I}*{9.4} for a nice discussion. Here we recall everything relevant to excision.
\\

Let $(T,\calO_T)$ now be a ringed topos. Let $U \in T$ be open with closed complement $Z$ and set $\calO_U := j^*\calO_T$ and $\calO_Z := i^*\calO_Z$. We have the following diagrams of topoi, where each arrow is left adjoint to the arrow directly below it. The functors $j^*$, and $j_*$ (resp. $i^{*}$ and $i_{*}$) restrict to a pair of adjoint functors (note that tensoring is not necessary!), giving a diagram
\[\xymatrix@C=2cm@M=.25cm{
	\Mod \calO_{U} 
		\ar@<7pt>[r]^{j^{\ab}_{!}}\ar@<-7pt>[r]_{j_{*}} &
	\Mod \calO_T
		\ar[l]|{j^{*}}
		\ar@<7pt>[r]^{i^{*}}\ar@<-7pt>[r]_{i^{!}} &
	\Mod \calO_{Z} 
		\ar[l]|{i_{*}}	
  }
\]
where the left arrows were defined in (\ref{p:modFunctors}) and the extra adjoint $i^!$ is defined by
\[
i^{!}P = \ker \left(i^{*}P \to i^{*}j_{*}j^{*}P\right)
\]
(see \cite{SGA4:I}*{ expos\'e 4, 9.5 and 14.4} for a more intrinsic description of $i^!$). Note that $i_*$ is thus exact as in the case of a closed immersion of schemes. 

The functor $j^{\ab}_!$ differs from the usual $j_!$ (see (\ref{p:abelianExtensionZero})), but when the context is clear we will write $j_!$; in particular $j_!$ of a sheaf of abelian groups will \emph{always} refer to $j^{\ab}_!$.

\begin{proposition}
\label{P:TopoiImmProp}

Let $P \in\Mod \calO_T$. Then the following are true.
    \begin{itemize}
        \item  [(i)] $0 \to j_{!}j^{*} P \to P \to i_{*}i^{*}P \to 0$ 
	is exact;
    
        \item  [(ii)] $0 \to i_{*}i^{!}P \to P \to j_{*}j^{*}P$ is
	exact;
    
        \item  [(iii)] $i_{*}i^{!}P$ is the `largest subsheaf of $P$
	supported on $Z$' (see \cite{SGA4:I}*{expos\'e 4, 9.3.5});
    
        \item  [(iv)] $i_{*}i^{*}P \cong i_{*}\calO_{Z} \otimes
	_{\calO_T} P$;
    
        \item  [(v)] $i_{*}i^{!}P \cong \HOM_{\calO_T}(i_{*}\calO_{Z},P)$.	
	\item  [(vi)] $j_{*}j^*P \cong
	\HOM_{\calO_T}\left(j_{!}\calO_{U},P\right)$.
    \end{itemize}

\end{proposition}

The proofs of these (and basically any identity involving these 6 functors) follows from a combination of the very simple description of these functors and maps between them via the covering theorem \cite{SGA4:I}*{14.3} and, for $M,N \in \Mod \calO$, the two adjunctions (or if one prefers, \emph{definitions}) $\Hom_T(-,\HOM_{\calO}(M,N)) \cong \Hom_{\calO}(M, \HOM_T(-,N))$ (as functors on $T$) \cite{SGA4:I}*{Proposition 12.1} and $\Hom_{\Ab T}(P, M \otimes_{\calO}N) \cong \Hom_{\calO}(M, \HOM_{\mathbb{Z}}(N,-))$ (as functors on $\Ab T$) \cite{SGA4:I}*{Proposition 12.7}.

\begin{proof}

These are in \cite{SGA4:I}*{ expos\'e 4}: (i \& ii) are 14.6 (account for the typo in (ii)), (iii) is 14.8, (iv) is 14.10, 1, (v) is 14.10, 2, and (vi) follows from the  proof of 14.10 (see also 12.6).

\end{proof}

\begin{definition}

Given $E \in \Mod \calO_T$, we define $\scrH_{Z}^{0}E := i_{*}i^{!}E$ to be the \defi{sheaf of sections of} $E$ \defi{supported on} $Z$ and denote the derived functors of $E \mapsto \scrH_{Z}^{0}E$  by $\scrH^{i}_{Z}E$.

\end{definition}

We can derive the functor $\scrH_{Z}^{0}$ either as a functor on $\Ab T$ or on $\Mod \calO_T$, because the functors $i_*$ and $i^!$ commute with the (exact) forgetful functor $\Mod \calO_T \to \Ab T$.

\begin{remark}

This is an appropriate name, because $\Gamma(T,\scrH_{Z}^{0}E)$ is the $\Gamma(T,\calO_T)$-module of all sections $s$ of $E$ supported on $Z$ (i.e., such that $s|_{U} = 0$); see \cite{SGA4:I}*{Proposition 14.8}.

\end{remark}

\begin{definition}
 
Let $f\colon (T,\calO_T) \to (T',\calO_{T'})$ be a morphism of ringed topoi. We define the \defi{cohomology} (resp. \defi{relative cohomology}) \defi{of} $E$ \defi{supported on}  $Z$ to be the right derived functors of $E \mapsto \Gamma(T,\scrH_{Z}^{0}E)$ by $H^{i}_{Z}(T,E)$ and the derived functors of $E \mapsto f_*\scrH_{Z}^{0}E$ by $\mathbb{R}f_{*,Z}E$.

\end{definition}

\begin{proposition}
\label{P:excisionRelations}

Let $E \in \Mod \calO_T$ and let the notation be as above. Then the following are true.
\begin{itemize}
   \item [(i)] $0 \to \scrH^{0}_{Z}E \to E \to j_{*}j^{*}E \to \scrH^{1}_{Z}E \to 0$ is exact.
    
   \item [(ii)] There is a long exact sequence 
     \[
	\ldots \to H^{i}_{Z}(T,E) \to H^{i}(T,E) \to H^{i}(U, j^{*}E) 
	\to H^{i+1}_{Z}(T,E)\to \ldots 
	\]

   \item [(iii)] Let $U' \subset U \subset T$ be a sequence of open immersions with closed complements $Z \hookrightarrow Z' \hookrightarrow T$ and denote by $Z' \cap U$ the closed complement of $U'$ in $U$. Then there is a long exact sequence 
     \[
	\ldots \to H^{i}_{Z}(T,E) \to H_{Z'}^{i}(T,E) \to H^{i}_{Z' \cap U}(U, j^{*}E) 
	\to H^{i+1}_{Z}(T,E)\to \ldots 
	\]
        
   \item [(iv)] There is a spectral sequence 
     \[
       \mathbb{R}^jf_* \scrH^j_Z E \Rightarrow \mathbb{R}^{i+j} f_{*,Z}E.
     \]
    \end{itemize}

\end{proposition}

\begin{proof}

Claims (i) - (iii) follow directly from Proposition \ref{P:TopoiImmProp} above; see \cite{SGA4:II}*{expos\'e 5, Proposition 6.5} for (i) and (ii), and for (iii) apply the proof of (ii) but with $P = i_*i^!\calO_T$ (instead of $P = \calO_T$) in Proposition \ref{P:TopoiImmProp} (ii). Claim (iv) is just the spectral sequence associated to a composition of derived functors.

\end{proof}

Lastly, we discuss functorality. For a ringed topos $(T,\calO_T)$ and a morphism of topoi $g\colon T' \to T$, we set  $\calO_{T'} := g^{*}\calO_T$, and if $T' = T_{/X}$ for some $X \in T$ we write $\calO_X$ for $\calO_{T_{/X}}$. 

\begin{proposition}
\label{P:closedSupportRelative}

Let $(T,\calO_T)$ be a ringed topos, let $f\colon X' \to X$ be a morphism in $\widehat{T}$, let $j\colon U \subset X$ be an open immersion with closed complement $i\colon Z \hookrightarrow T_{/X}$, let $j'\colon U' = U \times_X X' \subset X'$ and denote its closed complement by $i'\colon Z'\hookrightarrow T_{/X'}$.  Let $E \in \Mod \calO_X$. Then 
there is a natural map 
\[
  H^i_Z(T_{/X},E) \to H^i_{Z'}(T_{/X'},f^*E).
\]

\end{proposition}

\begin{proof}

It follows from the commutativity of the diagram 
\[
\xymatrix{
T_{/U'}\ar[d]\ar[r]&T_{/X'}\ar[d]&\\
T_{/U}\ar[r]&T_{/X}&
}
\]
of topoi that 
\[
f^*j_*j^*E \cong j'_*j'^*f^*E
\]
and that the composition 
\[
f^*E \to f^*j_*j^*E \cong j'_*j'^*f^*E
\]
is the adjunction. Then there is an isomorphism
\[
f^*i_*i^! E = f^* \ker \left(E \to j_*j^*E\right) = \ker \left(f^*E
  \to f^*j_*j^*E\right) = i'_*i^{'!}f^*E,
\]
where the second equality follows from exactness of $f^*$ on $\Mod
\calO_X$ (recall that $\calO_{X'} = f^*\calO_X$) and the other two follow from Proposition \ref{P:TopoiImmProp} (ii);
the natural map 
\[
H^0(T_{/X}, i_*i^!E) \to  H^0(T_{/X'},f^*i_*i^!E) \cong H^0(T_{/X'}, i'_*i^{'!}f^*E)
\]
thus induces a map 
\[
  H^i_Z(T_{/X},E) \to H^i_{Z'}(T_{/X'},f^*E).
\]

\end{proof}

\subsection{Excision on the overconvergent site}
\label{S:overconvergentExcision}

Here we apply the very general notions of Section \ref{S:excisionTopoi} to the overconvergent site. We state everything for $\AN^{\dagger}$ and for the sake of brevity omit restating definitions for the good variants on $\AN^{\dagger}_{\g}$, but note that everything carries over without incident.
\\

Let $X \to \AN^{\dagger}(\calV)$ be a fibered category over the overconvergent site and let $F \in \Mod \calO^{\dagger}_X$ be an overconvergent module on $X$. Let $U \in X_{\AN^{\dagger}}$ be open in the sense of Section \ref{S:excisionTopoi} (i.e.  $U$ is a subsheaf of the final object of $X_{\AN^{\dagger}}$). Denote by $j\colon  U_{\AN^{\dagger}} := (X_{\AN^{\dagger}})_{/U} \hookrightarrow  X_{\AN^{\dagger}}$ the open immersion of topoi induced by restriction, denote by $Z_{\AN^{\dagger}}$ the closed complement of $U$ in $X_{\AN^{\dagger}}$, and denote by $\calO^{\dagger}_U$ and $\calO^{\dagger}_Z$ the restrictions of $\calO^{\dagger}_X$ to $U_{\AN^{\dagger}}$ and $Z_{\AN^{\dagger}}$. As in Section \ref{S:excisionTopoi} this induces a collage of adjoint functors 
\[\xymatrix@C=2cm@M=.25cm{
	\Mod \calO^{\dagger}_{U} 
		\ar@<7pt>[r]^{j_{!}}\ar@<-7pt>[r]_{j_{*}} &
	\Mod \calO^{\dagger}_X
		\ar[l]|{j^{*}}
		\ar@<7pt>[r]^{i^{*}}\ar@<-7pt>[r]_{i^{!}} &
	\Mod \calO^{\dagger}_{Z} 
		\ar[l]|{i_{*}}.	
  }
\]

\begin{definition}

Let $E \in \Mod \calO_{X}^{\dagger}$ be an overconvergent module. We define $\underline{\Gamma}^{\dagger}_{Z}E := i_{*}i^{!}E$ to be the \defi{subsheaf of} $E$ \defi{of sections supported on} $Z$, and we define $H^{0}_{Z}(\AN^{\dagger} X,E) := H^{0}(\AN^{\dagger} X,\underline{\Gamma}^{\dagger}_{Z}E)$ to be the $H^{0}(\AN^{\dagger} X,\calO_{X}^{\dagger})$ submodule of \defi{sections of} $E$ \defi{supported on} $Z$. For a morphism $f\colon X \to Y$ of categories fibered over $\AN^{\dagger} \calV$, we define the \defi{relative cohomology} \defi{of} $E$ \defi{supported on}  $Z$ to be $\mathbb{R} f_{\AN^{\dagger}*}\underline{\Gamma}^{\dagger}_ZE$, which we denote by $\mathbb{R}f_{\AN^{\dagger}*,Z}E$. Since the realization functors are exact, when $Y = \Spec k$ the realization of the $i^{\Th}$ cohomology sheaf of $\mathbb{R}f_{\AN^{\dagger}*,Z}E$ is isomorphic to the $i^{\Th}$ derived functor of $H^{0}_{Z}(\AN^{\dagger} X,E)$; consequently we denote both of these $K$-vector spaces by $H^{i}_{Z}(\AN^{\dagger} X,E)$ and $H^{i}_{Z}(\AN^{\dagger} X,\calO^{\dagger}_X)$ by $H^{i}_{\rig,Z}(\AN^{\dagger} X)$.

\end{definition}

This differs slightly from the definition of \ref{S:excisionTopoi}, and in addition we have switched from the notation $\scrH_{Z}^{0}$ of \cite{SGA4:I} to the notation $\underline{\Gamma}^{\dagger}_Z$ of \cite{leStum:rigidBook}*{Definition 5.2.10}.

\begin{example}

We will mainly consider the following examples of open immersions of sites.
\begin{itemize}
\item [(i)] $\AN^{\dagger} U \subset \AN^{\dagger} X$ with $X$ a scheme over $k$ and $U \subset X$ an open subscheme. 
\item [(ii)] $\AN^{\dagger} (U,V) \subset \AN^{\dagger} (X,V)$ with $(X,V) \in \AN^{\dagger}\calV$ an overconvergent variety and $U \subset X$ an open subscheme. 
\item [(iii)] $\AN^{\dagger} (U_V) \subset \AN^{\dagger} (X_V)$ with $(X,V) \in \AN^{\dagger}\calV$ an overconvergent variety and $U \subset X$ an open subscheme, where $X_V$ is the image subpresheaf of the morphism of sheaves $(X,V) \to X$ (see Definition \ref{D:imagePresheaf}). 
\item [(iv)] More generally, for an overconvergent variety $(C,O)$ and a scheme $X$ over $k$ with a morphism $X \to C$, we can consider the relative variants $\AN^{\dagger} X/O$ and $\AN^{\dagger} X_V/O$ (see Definition \ref{D:overconvergentBase}).
\end{itemize}

\end{example}

These examples are all `representable' in the following sense.

\begin{definition}
\label{D:representable}

Let $j\colon U \subset X$ be a morphism of categories fibered over $\AN^{\dagger}(\calV)$ which induces an open immersion of topoi. We say that $j$ is \defi{representable} if for any overconvergent variety $(X',V')$ and morphism of fibered categories $(X',V') \to X$, there exists an open subscheme $U' \subset X'$ such that $(U',V')$ represents the 2-fiber product $U \times_ X \AN^{\dagger}(X',V')$.

\end{definition}

Here we rewrite the excision sequences from Subsection \ref{S:excisionTopoi}.

\begin{proposition}[Translation of Proposition \ref{P:excisionRelations}]
\label{P:excisionRelationsDagger}

    Let $j\colon U \subset X$ be a representable open immersion and let $E \in \Mod \calO^{\dagger}_X$ be an overconvergent module. Then with the notation above, the following are true.
    \begin{itemize}
        \item [(i)] $0 \to \underline{\Gamma}^{\dagger}_{Z}E \to E \to j_{*}j^{*}E \to 0$ is exact.
    
        \item [(ii)] There is a long exact sequence 
	\[
	\ldots \to H^{i}_{Z}(\AN^{\dagger} X,E) \to H^{i}(\AN^{\dagger} X,E) \to H^{i}(\AN^{\dagger} U, j^{*}E) 
	\to H^{i+1}_{Z}(\AN^{\dagger} X,E)\to \ldots 
	\]

        \item [(iii)] There is a spectral sequence 
        \[
         \mathbb{R}^jf_{\AN^{\dagger}*} \mathbb{R}^j\underline{\Gamma}^{\dagger}_Z E \Rightarrow \mathbb{R}^{i+j} f_{\AN^{\dagger},Z}E.
        \]

    \end{itemize}

\end{proposition}

\begin{proof}

Most of this is Proposition \ref{P:excisionRelations}; the only thing to check is that the map $E \to j_*j^*E$ is surjective. Since the morphism of sites of Definition \ref{D:sheafRealization} defines a
bijection of coverings, surjectivity can be checked on realizations. Let $(Y,V)$ be an overconvergent variety over $X$. Let $j'\colon U' \subset Y$ be an open immersion such that $(U',V)$ represents the fiber product $U \times_X (Y,V)$. Then $]j'[\colon ]U'[_V \hookrightarrow
\, ]Y[_V$ is now a closed immersion of analytic spaces.  Then, by the proof of part (i) of Proposition \ref{P:gammaRealization} below, there is an isomorphism 
\[
(j_*j^*E)_{Y,V} \cong \, ]j'[_*]j'[^* E_{Y,V}
\]
such that the composition
\[
E_{Y,V} \to (j_*j^*E)_{Y,V} \cong \, ]j'[_*]j'[^* E_{Y,V}
\]
is the adjunction
\[
E_{Y,V} \to ]j'[_*]j'[^* E_{Y,V}.
\]
By Proposition \ref{P:TopoiImmProp} (ii) (noting by Remark \ref{R:geometricImmersions} that $]j'[_*$ is a closed immersion of topoi) this map is surjective.

\end{proof}

The first task is to check that this agrees with the classical construction due to Berthelot of rigid cohomology supported in a closed subscheme\note{cite, maybe explain or define}. Most of the work is packaged into the following proposition.

\begin{proposition}
\label{P:gammaRealization}

Let $(C,O)$ be an overconvergent variety. Let $X \to \AN^{\dagger}(C,O)$ be a fibered category and let $U \subset X$ be a sub-fibered category of $X$ which is an open subtopos of $X_{\AN^{\dagger}}$ such that for all $(X', V')$, the fiber product $U \times_{X} (X', V')$ is isomorphic to $(U', V')$ with $U'\subset X'$ an open subscheme of $X'$. Denote the closed complement (defined in Section \ref{S:excisionTopoi}) of $U_{\AN^{\dagger}} \subset X_{\AN^{\dagger}}$ by $Z$. Let $E \in \Cris^{\dagger} X$ be an overconvergent module, let $(f,u)\colon (X'', V'') \to (X',V') \in \AN^{\dagger} X$ be a morphism of overconvergent varieties over $X$, and denote by $j'$ the inclusion $U' \hookrightarrow X'$, where $(U',V')$ is $U \times_X (X',V')$ (and similarly $j'' \colon U'' \hookrightarrow X'')$. Then the following are true.

\begin{enumerate}
\item [(i)] The realization $(\underline{\Gamma}^{\dagger}_Z E)_{X',V'}$ is canonically isomorphic to the kernel of the adjunction morphism
\[
\ker \left(E_{X', V'} \to  ]j'[_*]j'[^{\dagger}E_{X',V'}\right).
\]
\item [(ii)] Denote by $i' \colon W' \subset \, ]X'[_{V'}$ the (open) complement of the closed inclusion $]U'[_{V'} \hookrightarrow ]X'[_{V'}$. Then there is an isomorphism
\[
\left(\underline{\Gamma}^{\dagger}_ZE\right)_{X',V'} \cong i'_!i^{'-1}E_{X',V'}.
\]
\item [(iii)] The sheaf $\underline{\Gamma}^{\dagger}_Z E$ is a crystal.

\end{enumerate}

\end{proposition}

\begin{proof}

It suffices to consider the case $X = (X',V'),$ and $U = (U',V')$. For simplicity we drop a prime everywhere in the notation (i.e., we consider a morphism  $(X',V') \to (X,V) \in \AN^{\dagger} X$). To avoid the potentially awkward notation $\id_{\AN^{\dagger}}$ we denote the morphisms $(U,V)_{\AN^{\dagger}} \to (X,V)_{\AN^{\dagger}}$ and $(U',V')_{\AN^{\dagger}} \to (X',V')_{\AN^{\dagger}}$ by $j_{\AN^{\dagger}}$ and $j'_{\AN^{\dagger}}$. 

For (i), consider the diagram
    \[
    \xymatrix@C=2cm@M=.25cm{      
    \Cris^{\dagger}(X,V)
    	\ar[r]^{\varphi_{X,V*}}
	\ar[d]^{j_{\AN^{\dagger}}^{*}}& 
    \Mod(i^{-1}_{X}\,  \calO_{V})
	\ar[d]^{]j[_{V}^{*}}
\\
    \Cris^{\dagger}(U,V)
    	\ar[r]^{\varphi_{U,V*}}
	\ar[d]^{j_{\AN^{\dagger} *}}& 
    \Mod(i^{-1}_{U}\, \calO_{V})
	\ar[d]^{]j[_{V,*}}
\\
    \Cris^{\dagger}(X,V)
    	\ar[r]^{\varphi_{X,V*}}& 
    \Mod(i^{-1}_{X}\, \calO_{V})
    }.
    \]    
Since $E$ is a crystal the top square commutes, and the bottom square always commutes. Thus 
\[
(j_{\AN^{\dagger}*}j_{\AN^{\dagger}}^* E)_{X,V} \cong ]j[_*]j[^{\dagger} E_{X,V}
\]
and one can check, using the explicit descriptions of all relevant morphisms of topoi given in \cite{leStum:site}*{Section 2.3}, that under this isomorphism the realization of the adjunction is the adjunction; i.e., the composition
\[
E_{X, V} \to (j_{\AN^{\dagger}*}j_{\AN^{\dagger}}^* E)_{X,V} \cong\,  ]j[_*]j[^{\dagger}E_{X,V}
\]
is the adjunction morphism (alternatively this follows from the commutative diagram of the proof of \cite{leStum:site}*{3.2.1}). Finally, since the realization functor $\phi_{X,V*}$ is exact we conclude that 
\[
\left(\underline{\Gamma}^{\dagger}_ZE\right)_{X,V} = \left(\ker \left(E \to j_{\AN^{\dagger}*}j_{\AN^{\dagger}}^* E\right)\right)_{X,V} \cong \ker \left(E \to j_{\AN^{\dagger}*}j_{\AN^{\dagger}}^* E\right)_{X,V} 
\]
and by the above isomorphism this is $\ker \left(E_{X, V} \to  ]j[_*]j[^{\dagger}E_{X,V}\right)$, proving the first claim.

Claim (ii) follows from (i) since exactness of 
\[
0 \to i_!i^{-1} E_{X,V} \to E_{X,V} \to ]j[_*]j[^{-1}E_{X,V}
\]
can be checked on stalks, where it is clear. (Alternatively, this is a special case of the example of Remark \ref{R:geometricImmersions} and Proposition \ref{P:TopoiImmProp}).

Finally, by applying part (ii) twice, part (iii) amounts to showing that the natural map 
\[
u^*\left(\underline{\Gamma}^{\dagger}_ZE\right)_{X,V} \cong u^*i_!i^*E_{X,V} \to i'_!i'^*u^*E_{X,V} \cong \left(\underline{\Gamma}^{\dagger}_ZE\right)_{X',V'}
\]
induced by the isomorphism $u^*E_{X,V} \cong E_{X',V'}$ (where $i$ is the inclusion of the complement $i\colon W \subset \, ]X[_V$ of $]U[$ and $i'$ is the inclusion of the complement $i'\colon W' \subset ]X'[_{V'}$ of $]U'[$) is an isomorphism, which can be checked on stalks, where again it is clear.

\end{proof}

\begin{remark}

It is easy to see from the isomorphism $(\underline{\Gamma}^{\dagger}_ZE)_{X,V} \cong i'_!i^{'-1}E_{X,V}$ of Proposition \ref{P:gammaRealization} (ii) that $(\underline{\Gamma}^{\dagger}_ZE)_{X,V}$ is generally locally finitely presented and thus  $\underline{\Gamma}^{\dagger}_ZE$ is not locally finitely presented. This will make the comparison Theorem \ref{P:agreementClosedSupportsIsoc} more subtle, since we won't be able to apply Theorem \ref{T:cohomology}.

\end{remark}

\begin{remark}

Let $X/k$ be a scheme, let $U \subset X$ be an open scheme, and let $i\colon W \hookrightarrow X$ be the closed complement (as schemes) of $U$ in $X$. It is important to note that the closed complement $Z$ of $U_{\AN^{\dagger}} \subset X_{\AN^{\dagger}}$ (as topoi) is \emph{not} $W_{\AN^{\dagger}}$. In particular, for a module $E \in \Mod \calO^{\dagger}_X$, the module $\underline{\Gamma}^{\dagger}_ZE$ is not isomorphic to $i_{\AN^{\dagger}!}i_{\AN^{\dagger}}^*E$ or $i_{\AN^{\dagger}*}i_{\AN^{\dagger}}^*E$ and cannot be described in terms of $W_{\AN^{\dagger}}$.

\end{remark}

Let $(X,V)$ be a good overconvergent variety. Recall (see the discussion preceding Definition \ref{D:jDagger}) that the set $V_0$ of rigid points of $V$ naturally has the structure of a rigid analytic variety and that the inclusion $V_0 \hookrightarrow V$ induces an equivalence of categories
\[
\Coh \calO_{V_0} \cong \Coh \calO_V.
\]
We also defined (see Definition \ref{D:jDagger}) functors $j_X^{\dagger}$ (resp. $j_{X_0}^{\dagger})$  from $\Mod \calO_V$ (resp. $\Mod \calO_{V_0}$) to itself, which are isomorphic to the functors given by the formula
\[
E \mapsto \varinjlim j'_{*}j^{'-1} E 
\]
where the limit is taken over all neighborhoods $j'\colon V' \subset V$ of $]X[_V$ in $V$ (resp. strict neighborhoods $j'\colon V' \subset V_0$ of $]X[_{V_0}$ in $V_0$). 
\\

We define now the rigid analogue of the functor $\underline{\Gamma}^{\dagger}_Z$.

\begin{definition}[\cite{leStum:rigidBook}*{Definition 5.2.10}]
\label{D:jDaggerBerth}

Let $(X,V)$ be a good overconvergent variety and let $Z \hookrightarrow X$ be a closed subscheme with open scheme-theoretic complement $j\colon U \subset X$. Let $E_0 \in \Mod j_{X_0}^{\dagger} \calO_{V_0}$. We define the subsheaf $\underline{\Gamma}^{\dagger,\Ber}_Z E_0$ of $E_0$ of sections supported on $Z$ as the kernel  
\[
\ker\left(E_0 \to j_{U_0}^{\dagger} E_0\right).
\]

\end{definition}

\begin{proposition}
\label{P:agreementClosedSupports}

Let $S$ be a formal scheme over $\calV$ and suppose that $(S_k,S_K)$ is a good overconvergent variety. Let $p\colon X \to S_k$ be an algebraic variety over $S_k$ and let $X \subset P$ be an immersion of $X$ into a formal scheme $P/S$ such that $u\colon P \to S$ is smooth in a neighborhood of $X$  and proper at $X$ (see the paragraph before Definition \ref{D:realization}). Let $i\colon Z \hookrightarrow X$ be a closed subscheme with open scheme-theoretic complement $j\colon U \subset X$. Denote by $P_0 \hookrightarrow P_K$ the underlying rigid analytic variety of $P_K$. Let $E \in \Mod_{\fp} \calO^{\dagger}_{(X,P)}$ and $E_0 \in \Coh j^{\dagger}_{X_0}\calO_{P_0}$ such that there is an isomorphism  
\[
\phi\colon E_{X,P} \cong i_X^{-1}E_0^{\an}.
\]
Then $\phi$ induces an isomorphism
\[
(\underline{\Gamma}^{\dagger}_Z E)_{X,P} \cong i_X^{-1}(\underline{\Gamma}^{\dagger, \Ber}_Z E_0)^{\an}.
\]

\end{proposition}

\begin{proof}

We have a sequence of isomorphisms

\begin{align*}
 i_X^{-1}\left(\underline{\Gamma}_Z^{\dagger, \Ber}E_0\right)^{\an} 
  & :=  i_X^{-1} \left( \ker \left( 
          E_0 \to j_{U_0}^{\dagger}E_0
                  \right)\right)^{\an}\\
  & \cong i_X^{-1}  \ker \left( 
          E_0^{\an} \to \left(j_{U_0}^{\dagger}E_0\right)^{\an}
                  \right)\\
  & \cong  i_X^{-1}  \ker \left( 
          E_0^{\an} \to j_{U}^{\dagger}E_0^{\an}
                  \right)\\
  & \cong   \ker \left( 
          i_X^{-1}E_0^{\an} \to \, i_X^{-1}j_{U}^{\dagger}E_0^{\an}
                  \right)\\
  & \cong   \ker \left( 
          i_X^{-1}E_0^{\an} \to \, ]j[_*]j[^{\dagger}i_X^{-1}E_0^{\an}
                  \right)\\
  & \cong   \ker \left( 
          E_{X,P} \to \, ]j[_*]j[^{\dagger}E_{X,P}
                  \right)\\
  & \cong  \left(\underline{\Gamma}_Z^{\dagger}E\right)_{X,P}
\end{align*}
where the functors $i_X^{-1}$ and $(-)^{\an}$ commute with $\ker$ because they are left exact, and all other justifications (e.g., that the composition $E_0^{\an} \to \left(j_{U_0}^{\dagger}E_0\right)^{\an} \cong j_{U}^{\dagger}E_0^{\an}$ is the adjunction) follow from the explicit descriptions of each functor and isomorphism.

\end{proof}

Let $(C,O)$ be an overconvergent variety and let $X$ be an algebraic variety over $C$. Recall (see Definition \ref{D:MIC}) that we defined categories $\Strat^{\dagger}$, $\MIC$, $\MIC^{\dagger}$, and $\Isoc^{\dagger}$ and constructed natural maps
\[
\Cris^{\dagger}X_V/O \cong \Strat^{\dagger} i_X^{-1} \calO_V \to \MIC (X,V/O)
\] 
which induce an equivalence of categories
\[
\Mod^{\dagger}_{\fp}(X_V/O) \cong \MIC^{\dagger}(X,V/O).
\]
Let $(X,V)$ be an overconvergent variety and let $Z \hookrightarrow X$ be a closed subscheme with open scheme-theoretic complement $U$. Let $E \in \Cris^{\dagger}X_V/O$. Then by Proposition \ref{P:gammaRealization} (iii), $\underline{\Gamma}_Z^{\dagger}E$ is also a crystal, and so the realization $(\underline{\Gamma}_Z^{\dagger}E)_{X,V}$ admits a stratification. In fact, $\underline{\Gamma}_Z^{\dagger}E$ is a subsheaf of $E$ and the stratification of  $(\underline{\Gamma}_Z^{\dagger}E)_{X,V}$ is the restriction of the stratification on $E_{X,V}$. On the other hand, let $E_0 \in \Strat^{\dagger} i_{X_0}^{-1}\calO_{V_0}$. Then le Stum proves in \cite{leStum:rigidBook}*{Corollary 6.1.4} that $\underline{\Gamma}_Z^{\dagger, \Ber}E_0$ is stable under the stratification of $E_0$. This gives the following.

\begin{corollary}
\label{P:agreementClosedSupportsIsoc}

Under the assumptions of Proposition \ref{P:agreementClosedSupports} above, let $E \in \Mod_{\fp} \calO^{\dagger}_{}$ and $E_0 \in \Isoc^{\dagger}(X \subset \overline{X}/S)$ be such that there is an isomorphism 
\[
\phi\colon E_{X,P} \cong i_X^{-1}E_0^{\an}.
\]
Then the induced isomorphism 
\[
(\underline{\Gamma}^{\dagger}_Z E)_{X,P} \cong i_X^{-1}(\underline{\Gamma}^{\dagger, \Ber}_Z E_0)^{\an}
\]
of Proposition \ref{P:agreementClosedSupports} respects the stratifications (and thus the connections).

\end{corollary}

\begin{proof}

This is clear from the preceding construction since everything is functorial and the stratifications on $\underline{\Gamma}^{\dagger}_Z E$ and $\underline{\Gamma}^{\dagger, \Ber}_Z E_0$ are the restrictions of the stratifications on $E$ and $E_0$, which agree by \cite{leStum:site}*{Theorem 2.5.9}.

\end{proof}

We now come to the proof of our main theorem. 

\begin{proof}[Proof of Theorem \ref{P:agreementClosedSupportsCohomology}]

The exact sequences 
\[
0 \to \underline{\Gamma}^{\dagger,\Ber}_{Z}E_0 \to E \to j_{U_0}^{\dagger}E \to 0
\]
(\cite{leStum:rigidBook}*{Lemma 5.2.9}) and 
\[
0 \to \underline{\Gamma}^{\dagger}_{Z}E \to E \to j_{*}j^{*}E \to 0
\]
(Proposition \ref{P:TopoiImmProp} (ii)) induce a pair of morphisms of exact triangles
\[
\xymatrix{
\left(\mathbb{R}p_{\rig}\underline{\Gamma}^{\dagger, \Ber}_Z E_0\right)^{\an} 
  \ar[r] \ar[d] &
\left(\mathbb{R}p_{\rig}E_0\right)^{\an}
  \ar[r] \ar[d] &
\left(\mathbb{R}p_{\rig}j_{U_0}^{\dagger} E_0\right)^{\an}
  \ar[d] \\
\mathbb{R}u_{K*} \underline{\Gamma}^{\dagger}_Z E_{X,P} \otimes i_X^{-1}\Omega^{\bullet}_{P_K/S_K}
 \ar[r] &
\mathbb{R}u_{K*}  E_{X,P}\otimes i_X^{-1}\Omega^{\bullet}_{P_K/S_K}
 \ar[r] &
\mathbb{R}u_{K*} ]j[_*]j[^{\dagger} E_{X,P}\otimes i_X^{-1}\Omega^{\bullet}_{P_K/S_K}
\\
\left(\mathbb{R}p_{\AN^{\dagger}_*}\underline{\Gamma}^{\dagger}_ZE\right)_{S_k,S_K}
  \ar[r] \ar[u] &
\left(\mathbb{R}p_{\AN^{\dagger}_*}E\right)_{S_k,S_K}
  \ar[r] \ar[u] &
\left(\mathbb{R}p_{\AN^{\dagger}_*}j_*j^*E\right)_{S_k,S_K}
  \ar[u] 
}
\]
where the top vertical arrows are defined as in Theorem \ref{T:cohomology} and the bottom vertical arrows are defined as in \ref{T:deRham}. The vertical arrows of the middle column are isomorphisms by Theorems \ref{T:cohomology} and \ref{T:deRham}, and since the functors $j_*$, $]j[_*$, and $j_{U_0}^{\dagger}$ are exact, the right vertical arrows are also isomorphisms (again by Theorems \ref{T:cohomology} and \ref{T:deRham}). By the five lemma, the left column consists of quasi-isomorphisms too. The excision statement is clear since the excision long exact sequences are the long exact sequences associated to these exact triangles.

\end{proof}

\begin{proposition}[Functorality]
\label{P:functoralityClosedSupport}

Let $(C,O)$ be an overconvergent variety. Let $f\colon X' \to X$ be a morphism of schemes over $C$, let $i\colon Z \hookrightarrow X$ be a closed substack with open stack-theoretic complement $j \colon U \subset X$, and let $Z' = Z \times_X X'$ and $U' = U \times_X X'$, where we denote the inclusions into $X'$ by $i'$ and $j'$. Let $E \in \Mod \calO^{\dagger}_{X/O}$ be an overconvergent module. Then there is a natural map 
\[
H^i_Z(X_{\AN^{\dagger}}, E) \to H^i_{Z'}(X'_{\AN^{\dagger}}, f^*_{\AN^{\dagger}}E);
\]
in particular (setting $E = \calO^{\dagger}$), the assignment 
\[
(Z \hookrightarrow X) \mapsto H^i_{\rig,Z}(X_{\AN^{\dagger}})
\]
is a contravariant functor from the category of closed immersions of stacks (with morphisms cartesian diagrams) to the category of $K$-vector spaces.

\end{proposition}

Of course, for a map $X/O \to T$, with $T$ a fibered category over $\AN^{\dagger}(C,O)$, there is a similar map of relative cohomology with supports in $Z$. Also, the same statement holds if we replace everything by its good variant.

\begin{proof}

This is just a translation of Proposition \ref{P:closedSupportRelative} to the overconvergent site.

\end{proof}

\begin{remark}

Using the notation of Proposition \ref{P:agreementClosedSupports}, let $E \in \Mod_{\fp} \calO^{\dagger}_{(X,P)}$ and let $E_0 \in \Isoc^{\dagger}(X \subset \overline{X}/S)$. Then the techniques used in Proposition \ref{P:agreementClosedSupports} to show that $\left(\underline{\Gamma}^{\dagger}_Z E\right)_{X,P} \cong i_X^{-1} \left(\underline{\Gamma}_Z^{\dagger, \Ber}E_0\right)^{\an}$ also show that this functorality map agrees with the classical functorality map of Berthelot \cite{leStum:rigidBook}*{6.3.5} -- both arise from the very general constructions of \cite{SGA4:I} and again the work is to show that the adjunctions used in Proposition \ref{P:closedSupportRelative} match up in both contexts.

\end{remark}

\section{Overconvergent cohomology with closed supports and stacks}
\label{S:main}
 
Let $C$ be a site and let $u \colon D \to C$ be a fibered category. Suppose moreover that every arrow of $C$ is cartesian; it then follows that that $u$ commutes with fiber products. Recall from (\ref{E:cocontinuous}) that if we endow $D$ with the induced topology, then the functor $u$ is then both continuous and cocontinuous. In particular, by (\ref{E:cocontinuous}) we get a triple of morphisms
 \[
  \xymatrix@C=2cm@M=.25cm{
	\widetilde{D}
		\ar@<7pt>[r]^{u_{!}}\ar@<-7pt>[r]_{u_*} &
	\widetilde{C}
		\ar[l]|{u^{*}}.
  }
  \]
such that each arrow is left adjoint to the arrow below. Since $u^*$ has both a left and right adjoint it is exact and thus the pair $(u^*,u_*)\colon \widetilde{D} \to \widetilde{C}$ defines a morphism of topoi. On the other hand, $u_!$ is generally not exact (almost any non-trivial example will exhibit this). Finally, we remark that $u_!$ does not take abelian sheaves to a abelian sheaves; nonetheless $u^*\colon \Ab \widetilde{C} \to \Ab \widetilde{D}$ has a \emph{different} left adjoint, $u_!^{\ab}$, which we also denote by $u_!$ when there is no confusion (see (\ref{p:abelianExtensionZero})). In particular, $u^*$ takes injective abelian sheaves to injective abelian sheaves.
\\


\defi{Algebraic Spaces and Stacks}: We refer to \cite{Knutson:algebraicSpaces} and \cite{LaumonMB:champs} for basic definitions regarding algebraic spaces and stacks. Note in particular the standard convention that a representable morphism of stacks is represented by \emph{algebraic spaces}. Actually, most of the theory works for arbitrary fibered categories, but particular examples and theorems will require algebraicity and finiteness assumptions, which will be clearly stated when necessary.

\begin{definition}
\label{D:overconvergentStack}

Let $\calX \to \Sch_k$ be a fibered category. We define the \defi{overconvergent site} $\AN^{\dagger}(\calX)$ of $\calX$ to be the category $\AN^{\dagger}(\calV) \times_{\Sch_k}\calX$ with the topology induced by the projection $\AN^{\dagger}(\calX) \to \AN^{\dagger}(\calV)$. We define the \defi{good overconvergent site} $\AN^{\dagger}_{g}(\calX)$ similarly.

\end{definition}

\begin{remark}

Concretely, an object of $\AN^{\dagger}(\calX)$ is an overconvergent variety $(X,V)$ together with an object of the fiber category $\calX(X)$; by the 2-Yoneda lemma this data is equivalent to an overconvergent variety $(X,V)$ and a map of  categories $\Sch_X \to \calX$ fibered over $\Sch_k$. In particular, for a presheaf $T$ on $\Sch_k$ with associated fibered category $\Sch_T \to \Sch_k$, $\AN^{\dagger}(T)$ defined as before is equivalent to $\AN^{\dagger}(\Sch_T)$.

\end{remark}

\begin{definition}
\label{D:isocrystalsStack}

Let $\calX$ be a fibered category over $k$. We define the \defi{sheaf of overconvergent functions} $\calO^{\dagger}_{\calX_{\g}}$ to be the pullback $u^{-1}\calO^{\dagger}_{\calV_{\g}}$ with respect to the projection $u\colon\AN^{\dagger}_{\g}(\calX) \to \AN^{\dagger}_{\g}(\calV)$.

\end{definition}

Of course, the main case we consider will be when $\calX$ is an algebraic stack over $k$. As before we will consider the categories $\Mod \calO^{\dagger}_{\calX_{\g}}$, $\Mod_{\fp} \calO^{\dagger}_{\calX_{\g}}$ and $\Cris^{\dagger} \calX_{\g}$. A map $f\colon \calX \to \calY$ of fibered categories induces a morphism $f_{\AN^{\dagger}_{\g}} \colon \calX_{\AN_{\g}^{\dagger}} \to \calY_{\AN_{\g}^{\dagger}}$ of topoi, and to any abelian sheaf $F \in \calX_{\AN_{\g}^{\dagger}}$ and for any good overconvergent variety $(X,V) \in \AN^{\dagger}_{g} \calX$ one can study the cohomology $\mathbb{R}f_{\AN^{\dagger}_{g}*}F$ and the realization $F_{X,V}$.

\begin{remark}

Functorality and excision for stacks (Propositions \ref{P:functoralityClosedSupport}, \ref{P:excisionRelationsDagger}) follow formally as in Subsection \ref{S:overconvergentExcision}.  On the other hand, Theorem \ref{P:agreementClosedSupportsCohomology} (agreement) doesn't make sense for stacks (because there is no classical cohomological object to compare with overconvergent cohomology).

\end{remark}

\begin{remark}

It is worth noting here that the usual issues of functorality surrounding the lisse-\'etale site (e.g. functorality of crystalline cohomology for stacks \cite{Olsson:Crystal}) are not issues here.

\end{remark}

We end by applying cohomological descent \cite{Zureick-Brown2014}*{Theorem 1.1} to prove the finiteness of rigid cohomology with support in a closed subscheme.

\begin{proposition}
\label{P:finiteness}

Let $f\colon X \to \Spec k$ be a separated algebraic stack of finite type over $k$ and let $Z \subset X$ be a closed substack. Then for every $i \geq 0$, $H_Z^i(\AN^{\dagger}_{\g} X)$ is a finite dimensional $K$-vector space. 

\end{proposition}

\begin{proof}

First we do the case without supports (i.e., $Z = X$). Let $p_0\colon X' \to X$ be a projective surjection from a scheme $X'$ which is quasi-projective over $k$ (which exists by \cite{Olsson:Hypercover}*{Theorem 1.1}), and as usual denote by $p_i\colon X'_i \to X$ the $(i+1)$-fold fiber product of $p_0$. Then by \cite{Zureick-Brown2014}*{Corollary 3.4 and Remark 3.3} (noting that by definition $p_i^*\calO^{\dagger}_{X_{\g}} = \calO^{\dagger}_{X_{i,\g}})$ there is a spectral sequence
\[
H^j(\AN^{\dagger}_{\g} X'_i) \Rightarrow H^{i+j}(\AN^{\dagger}_{\g} X).
\]
When $X$ is an algebraic space, $X'_i$ is a scheme, so by Theorem \ref{T:cohomologyCorollary} (noting that $X'_i$ is quasi-projective and thus the structure morphism is realizable), there is an isomorphism
\[
H^j(\AN^{\dagger}_{\g} X'_i) \cong H^j_{\rig}(X'_i)
\]
which is finite dimensional by \cite{kedlaya:finitenessCoefficients}*{Theorem 1.2.1}, so by the spectral sequence $H^i(X)$ is finite dimensional as well. Now that we know the result for an algebraic space, the case of $X$ a stack follows directly from the spectral sequence. Finally, the case with support in $Z$ follows from the excision exact sequence of Proposition \ref{P:excisionRelationsDagger}.

\end{proof}

\begin{remark}

Classically, many results only hold for the category $F\operatorname{-Isoc}^{\dagger}(X \subset \overline{X})$ of isocrystals with Frobenius action (see \cite{leStum:rigidBook}*{Definition 8.3.2}).  One can define an analogue on the overconvergent site, and the same argument will show that the cohomology of an $F$-isocrystal will be finite dimensional.

\end{remark}

\end{document}